\documentclass[12pt]{amsart}
\usepackage{amsmath}
\usepackage{amssymb}
\usepackage[all]{xy}
\usepackage{ytableau}
\usepackage{graphicx}
\usepackage{booktabs}
\usepackage{hyperref}
\usepackage{mathrsfs}

\usepackage[a4paper, total={6in, 8in}]{geometry} 

\newtheorem{lm}{Lemma}[section]

\newtheorem{theorem}[lm]{Theorem}

\theoremstyle{definition}
\newtheorem{defn}[lm]{Definition}

\newtheorem{ex}[lm]{Example}

\newcommand{\mdef}{\refstepcounter{lm}\medskip\noindent{\bf (\thelm) } }

\newcommand{\Z}{\operatorname{\mathbb{Z}}\nolimits}

\newcommand{\ignore}[1]{}

\newcommand{\mat}[1]{\left( \begin{array}{#1} }
\newcommand{\tam}{\end{array} \right) }

\newcommand{\N}{{\mathbb N}}

\newcommand{\ppm}[1]{\textcolor{green}{#1}}
\newcommand{\kb}[1]{\textcolor{orange}{#1}}
\newcommand{\kub}[1]{\textcolor{blue}{#1}}

\newcommand{\beq}{\begin{equation}}
\newcommand{\eq}{\end{equation}}

\theoremstyle{definition}

\newcommand{\tre}{triangular equivalence}

\newcommand{\mlem}[1]{\begin{lm} #1 \end{lm}}
\newcommand{\footnot}[1]{}  
\newcommand{\PP}{\mathscr P}

\title
{A generalised Euler--Poincar\'e formula for associahedra.
}
\author{Karin Baur}
\address{Department of Mathematics and scientific computing, 
University of Graz, Nawi Graz, 8010 Graz, Austria}
\author{Paul P. Martin}
\address{Department of Pure Mathematics, University of Leeds, 
Leeds LS2 9JT, UK}

\newcommand{\monthword}[1]{\ifcase#1\or January\or February\or March\or April\or May\or 
June\or July\or August\or September\or October\or November\or December\fi}
\date{\monthword{\the\month} \the\day, \the\year } 

\begin{document}

\maketitle

\newcommand{\EP}{Euler--Poincar\'e}

\begin{abstract}
We derive a formula for the number of
flip-equivalence classes of
tilings
of an $n$-gon
by collections of tiles 
of shape  dictated by an integer partition $\lambda$. 
%
  The proof uses the \EP\ formula; and the formula itself
  generalises the \EP\ formula for associahedra.
\end{abstract}

\medskip


\section{Introduction} \label{ss:intro}

In \cite{FominZel-I-02, FominZel-II03,FominZel03} 
Fomin and Zelevinsky introduce cluster algebras and in particular
find a
Lie theoretic
manifestation 
of the associahedron or Stasheff polytope \cite{Stasheff63}. 
In \cite{Chapoton2004} Chapoton investigates the combinatorics of the 
associahedron from the Lie theoretic perspective; while in \cite{Postnikov}, 
\cite{Scott}, Postnikov and Scott develop these constructions in 
the direction of total positivity and 
Grassmannians.
(The associahedron also
has a wider interconnected network of other uses.
See for example
\cite{Tonks97}, 
\cite{GKZ08}, \cite{Brown09}, \cite{Postnikov09}.) 
Baur et al.~\cite{BaurKingMarsh} develop Fomin-Zelevinsky's and Scott's 
original construction with 
emphasis on quiver mutation and cluster categories associated with the Grassmannian. 
In \cite{BaurMartin15}
we  
generalised Scott's map (\cite[Section 3]{Scott}) 
to the whole associahedron. 
A partial result on counting
tilings with prescribed tile sizes 
up to `flip equivalence' (cf. \cite{Hatcher91})
is included in \cite{BaurMartin15}.
Based on this, Baur-Schiffler-Weyman conjectured a general formula
\cite{BSW}.
In this 
paper, we
note that the formula also generalises the Euler--Poincare formula for the
associahedron, and 
prove the conjecture. 
The proof uses the polytopal property of the associahedron and 
the Euler-Poincar\'e formula
itself
among other devices. 



\begin{figure}
\includegraphics[width=1.4in]{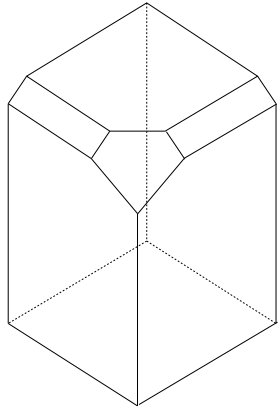}  
\caption{\label{fig:K5}
Associahedron $K_5$ drawn as a manifest convex polytope.}
\end{figure}

\newcommand{\betam}[1]{m_{#1}(\mu)}
\newcommand{\Pip}[2]{\Pi^{#1}_{#2}}
\newcommand{\coup}{\cup}
\newcommand{\triangula}{triangular}  
\newcommand{\type}{shape}
\newcommand{\shape}{type/shape}
\newcommand{\weight}{weight}
\newcommand{\fillup}[1]{\overline{#1}}

\newcommand{\sh}{{\mathsf s}} 
\newcommand{\ff}{{\mathsf f}}  

\newcommand{\OF}{O\!F} 

\newcommand{\len}[1]{l(#1)}  
\newcommand{\modl}[1]{\stackrel{\ \len{#1}}{\models}} 
\newcommand{\modll}[1]{{\models}^{\!\!\!\!#1}} 



\medskip

Let $P$ be a convex polygon with $n$ vertices, labeled $1,2,...,n$.
Here $A_n$ denotes the set of tilings of such an $n$-gon,
or equivalently the set of non-crossing subsets of diagonals $[i,j]$ of $P$. 
For example
\[
A_4 = \{ 
\raisebox{-.071in}{\includegraphics[height=.21in]{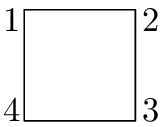}},
\raisebox{-.071in}{\includegraphics[height=.2in]{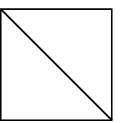}},
\raisebox{-.071in}{\includegraphics[height=.2in]{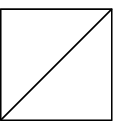}}
\}
= \{ \emptyset, \{ [1,3] \}, \{[2,4] \}\}
\]
There are two ways of triangulating a 4-gon. The move between these
two is called `flip'.
Two tilings are
{\em \triangula\ equivalent}, or {\em flip equivalent},  if 
they are related by any sequence of flips (for example any two
triangulations are equivalent \cite{Hatcher91}).
The set of classes of tilings under equivalence is $\AE_n$. 
These are enumerated in \cite{BaurMartin15}. 


It is convenient to arrange $A_n$  into a polytope, called the
associahedron \cite{Lee89}.
This is {\it ab initio} an abstract cell complex:
each tiling $t$ is a cell, and the boundary set of a cell is the set of
tilings obtained from $t$ by adding one more diagonal.
Thus in particular the vertices (0-cells)
are the tilings that are triangulations.
Removing a diagonal from such a tiling gives a tiling with one
quadrilateral tile.
Starting now from this tiling, 
the quadrilateral tile can be triangulated in two ways.
The tiling is an edge (1-cell) of the complex and
its boundary is the two completions obtained by triangulating the
quadrilateral. And so on.


\begin{theorem}[Haiman, Lee \cite{Lee89}; see e.g. {\cite[Theorem 3.39]{DevORourke}}]
  \label{th:KisP}
  There exists a convex $n$-dimensional polytope $K_{n+1}$ 
  called the
  associahedron whose vertices and edges form the flip graph of a
  convex $(n+3)$-sided polygon. The $k$-dimensional faces ($k$-cells) of this
  polytope are in one-to-one correspondence with the diagonalizations (tilings)
  of the polygon using exactly $n-k$ diagonals. 
\end{theorem}

See e.g. Figure~\ref{fig:K5}. 

The numbers $a_n = | A_n |$ are the little Schr\"oder numbers 
(see e.g. \cite{ps2000}).
Also of interest are the $f$-vectors \cite{DevORourke} of these polytopes \cite{Chapoton2004},
see Sloane's OEIS number A001003 \cite{SloaneOEIS} and cf.~(\ref{eq:triangles-case}).

\medskip

\mdef 
Let $\PP_n$ be the set of integer partitions of $n$, 
and $\PP$ be the set of all integer partitions.
We use partition notation as in Macdonald \cite[\S I.1]{Macdonald15},
as follows. 
For 
$\lambda\in\PP$ let $m_i(\lambda)$ denote the number of parts $i$.
We write $\lambda \in \PP_n$   
as $\lambda=(\lambda_1,\lambda_2,\dots,\lambda_r)$ where the $\lambda_i$ are 
non-negative integers in non-increasing order;
or in {\em exponential} form,
$\lambda=(1^{m_1},2^{m_2},\dots,r^{m_r},\dots)$.
We write $\lambda\coup\mu$ for the { combination of
  partitions}, thus $(3,2)\coup (4,2,1) =(4,3,2,2,1)$.
We write 
$l(\lambda)$ for the number of non-zero parts, the {\em length}; 
$|\lambda| = \sum_i\lambda_i =n$, the {\em \weight\ of  $\lambda$};
and we may write  $\lambda \in \PP_n$ as $\lambda \vdash n $. 

\bigskip 

\mdef \label{de:sh}
For tiling $t \in A_n$ let $m_i(t)$ denote 
the number of $(i+2)$-gonal tiles.
Define a `\type' function $\sh: A_n \rightarrow \PP$ by
$\sh(t) =  (1^{m_1(t)},2^{m_2(t)},\dots)$.
Let $A_n(\lambda) \subset A_n$ denote the subset of tilings of
\type\ $\lambda$. Thus:
\[
A_n = \bigcup_{\lambda \in \PP_{n-2}} A_n(\lambda)
\] 
The first few of the numbers $a_n(\lambda) = | A_n(\lambda) |$ are
computed in \cite{BaurMartin15}. 


Note that \type\ is  well-defined on \tre\ classes. 
Let $\AE_n(\lambda)$ denote the set, and 
$\ae_n(\lambda)$ denote the number of tilings of $P$ of 
a given \type\ up to triangular equivalence.

\newcommand{\hue}{hue} 


For $\mu \in \PP$ define $\mu^+ \in \PP$ by 
\beq \label{eq:mup}
\mu^+=(2^{m_1(\mu)},3^{m_2(\mu)}, \dots, (i+1)^{m_i(\mu)},\dots)
\eq

We extend the combination $\coup$ of partitions as follows. 
Let $a>0$, $\lambda\in\PP$. Define $\lambda\coup (1^{-a})$ as 
$(1^{m_1(\lambda)-a},2^{m_2(\lambda)}, 3^{m_3(\lambda)},\dots)$. 
For $m_1(\lambda)\ge a$, $\lambda\coup (1^{-a})$ is a partition. 
For example, $(1^3,2^2,3)\coup (1^{-2})=(1,2^2,3)$. 
We will consider tilings of \type\ $\lambda\coup (1^{-a})$, for some 
$\lambda$.
If $m_1(\lambda)<a$, $\lambda\coup (1^{-a})$ is not a partition,
and $A_n(\lambda\coup (1^{-a}))=\emptyset$.

\begin{theorem}\label{thm:general2}
For 
$\lambda,  
\mu\in\PP$, 
let
\[
\Pip{\lambda}{\mu} = \prod_{s\ge 1}{m_{s+1}(\lambda) + \betam{s} \choose \betam{s}} 
 = \prod_{s\ge 2}{m_{s}(\lambda\coup\mu^+) \choose m_{s}(\mu^+)}
 \]
Then 
\begin{equation}\label{eq:general}
\ae_n(\lambda)
 = \sum_{m=0}^{m_1(\lambda)-1}(-1)^m\sum_{\mu\,\vdash m} \,
       \Pip{\lambda}{\mu} \,
        a_n(\lambda\coup \mu^+\coup 1^{-|\mu^+|}) 
\end{equation}
\end{theorem}


%
%


Before proving the theorem we focus on one class of special cases:
the cases where $\lambda$ is of form $1^{n-2}$.
In these cases we know $\ae_n(\lambda) = 1$ by Hatcher's theorem (a) in
\cite{Hatcher91}.  
On the other hand the right hand side also takes a relatively simple
form in these cases. 
%
All coefficients $\Pip{(1^{n-2})}{\mu}$ on the right hand 
side are $1$ since $m_s(1^{n-2})=0$ for all $s>1$. The right hand side is thus 
\begin{equation}\label{eq:triangles-case}
\sum_{m=0}^{n-3}(-1)^m\sum_{\mu\,\vdash m} \,
               a_n(1^{n-2}\coup \mu^+\coup 1^{-|\mu^+|}) 
\end{equation}
Noting that $\sum_{\mu\,\vdash m} \, a_n(1^{n-2}\coup \mu^+\coup 1^{-|\mu^+|})$ 
is the number of $m$-cells of $K_{n-1}$, 
since the signs alternate we actually have the 
Euler--Poincar\'e formula.
Thus this special case of the theorem asserts:
\[
1 = \mbox{Euler-Poincar\'e characteristic of $K_{n-1}$}
\]

Consider the Schl\"afli--Poincar\'e Theorem:
\begin{theorem}[Schl\"afli, Poincar\'e; see e.g.\ \cite{ZieglerBl07}] \label{th:SP}
  Let $f_i$ denote the number of $i$-cells of the $d$-dimensional polytope $P$.
Then the Euler-Poincar\'e formula states that
\[
\sum_{i=0}^d (-1)^i f_i =1  . 
\]
\end{theorem}
%

\noindent
By Theorem 1.1, $K_{n-1}$ is a 
$n-3$-dimensional polytope, so we can apply Theorem~\ref{th:SP}. We thus 
see 
on the one hand that  Theorem~\ref{thm:general2} is verified in the case
$\lambda = 1^{n-2}$;
and on the other hand 
that the special cases of the theorem
coincide with 
the Euler-Poincar\'e formulae for the various associahedra.
In other words, the `count' of triangulations of $P$ up to flip
equivalence computes the \EP\ characteristic of the corresponding
associahedron.

\section{Proof of main Theorem}
\subsection{Constructions and notation for proof: integer partitions} $ \; $ 

\mdef \label{de:sea}
(To take arms against a sea of triangles \cite{Shakespeare}.)
Observe that the map $\mu \mapsto \mu^+$ from (\ref{eq:mup}) corresponds to adding a
leading column (of length $l(\mu)$) to the Young diagram of $\mu$:
\[
\includegraphics[width=.5in]{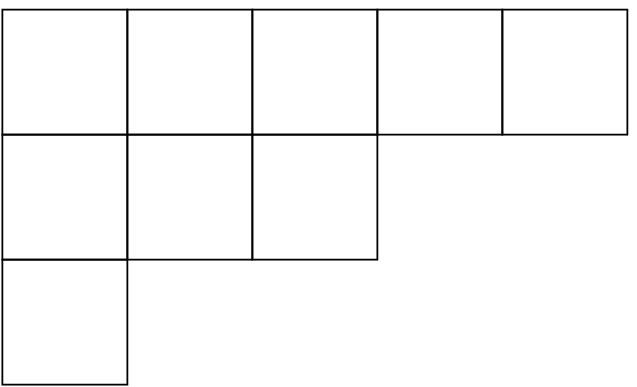} \mapsto \;\;
\includegraphics[width=.598in]{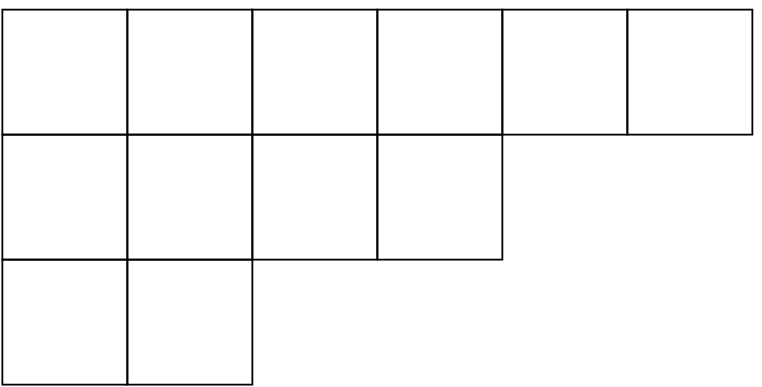} 
\]
(NB $\emptyset^+ = \emptyset$).
Assuming $\mu \neq \emptyset$  
we also define  $\mu \mapsto \mu^-$ by {\em removing } the first
column:
\beq \label{eq:mum}
\includegraphics[width=.3135in]{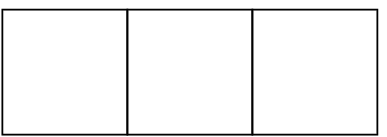} \mapsto \;
\includegraphics[width=.2078in]{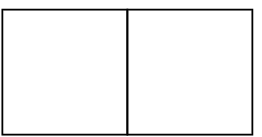} 
\hspace{.75in}
\includegraphics[width=.2006135in]{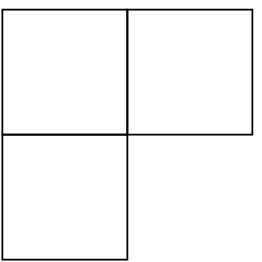} \mapsto \;
\includegraphics[width=.100912078in]{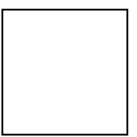} 
\hspace{.75in}
\includegraphics[width=.100912078in]{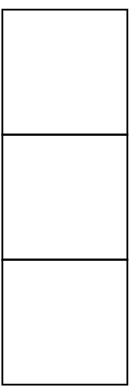} 
\mapsto \emptyset
\eq
Note $\mu \mapsto \mu^+ \mapsto (\mu^+)^- = \mu$, but 
$\mu \mapsto \mu^- \mapsto (\mu^-)^+ $  
is not $\mu$ unless $m_1(\mu)= 0$.

At the level of {\em tilings} $t$ of \type\ $\sh(t) = \mu$
(as in  (\ref{de:sh})), $\mu \mapsto \mu^-$
`forgets' all triangles. It nominally replaces every other $i$-gon with
an $(i-1)$-gon. Thus $(\mu^-)^+$ simply forgets all triangles and
preserves other gons.

Note that if $\mu \in \PP_n$ then $\mu^-$ lies in some $\PP_{n-m}$
(specifically in $\PP_{n-l(\mu)}$).
Indeed $\mu \mapsto \mu^-$ defines an injective map
\beq \label{eq:PiP}
\PP_n \hookrightarrow  \cup_{m=0}^{n-1} \PP_{m}
                      =\cup_{m=1}^{n} \PP_{n-m}
\eq
(See (\ref{eq:mum}) for the examples in $\PP_3$.)


At the level of {\em $i$-cells} $t$ of associahedron $K_{n-1}$
then cell dimension
$$
i = | \sh(t)^- | .
$$
Note that this verifies the \EP\ formula assertion 
immediately 
after (\ref{eq:triangles-case}). 



\mdef \label{rem:A-l-m}
We define another 
function from tilings to partitions
(cf. (\ref{de:sh})). 
A tiling $t$ gives an integer partition $\ff(t)$ as follows. 
First consider the partition 
$\ff^+(t)$ of the number of its triangles 
into the unordered list of sizes of maximal triangulated regions. Then 
$\ff(t) = \ff^+(t)^-$.
Again this is well-defined on classes. \\
Examples: 
$$
\ff^+\left( 
\raisebox{-.121in}{
\includegraphics[width=1cm]{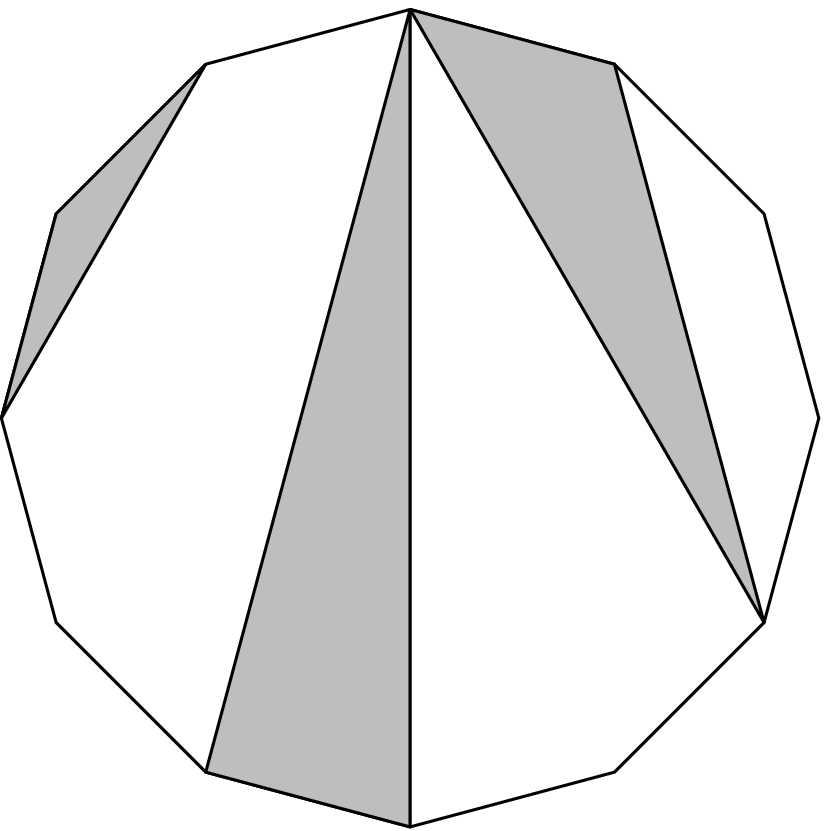}}
\right) = 
\ff^+\left( 
\raisebox{-.121in}{
\includegraphics[width=1cm]{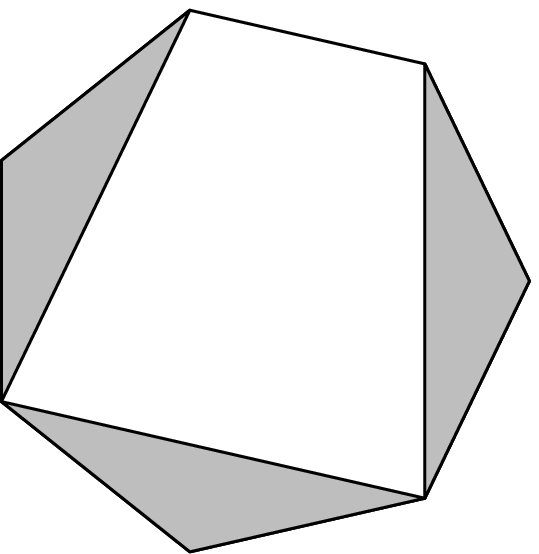}}
\right) = 
\ff^+(\includegraphics[width=1cm]{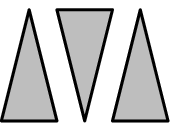})=(1,1,1)
%
; \hspace{1cm} 
\ff(\includegraphics[width=1cm]{figs/111.eps})=(0,0,0)=\emptyset .
$$ 
$$
\ff^+\left( 
\raisebox{-.121in}{
\includegraphics[width=1cm]{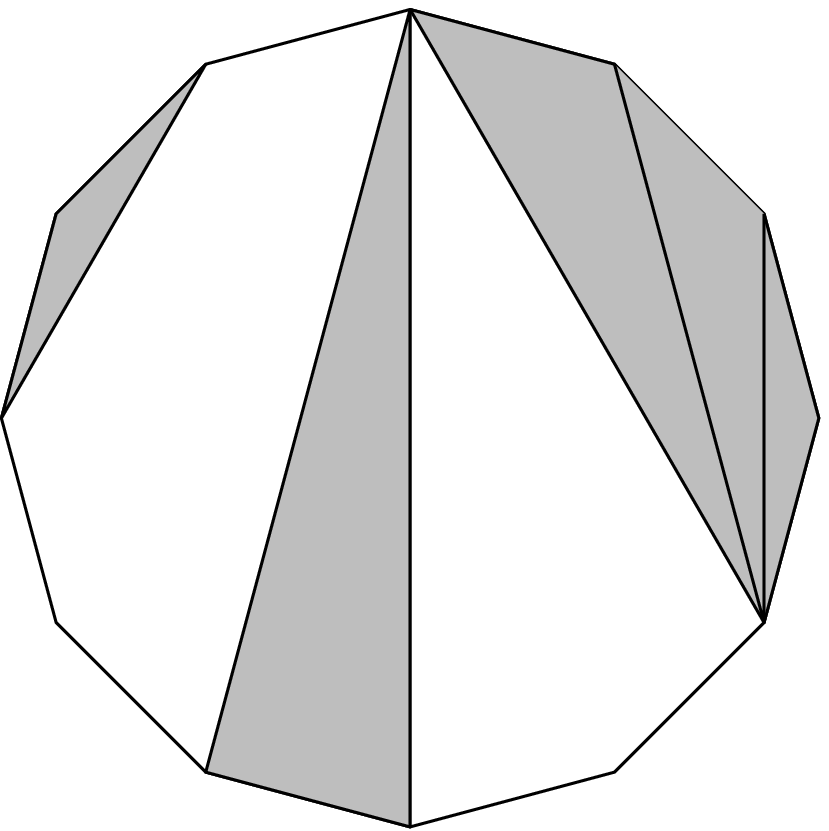}}
\right) = (3,1,1);
\hspace{1in}
\ff\left( 
\raisebox{-.121in}{
\includegraphics[width=1cm]{figs/12g113.eps}}
\right) = (2,0,0) = (2)
$$ 
\ignore{{
and 
$f^+(\includegraphics[width=.3cm]{pentagon.eps}\triangle)=(3,1)$ and 
$
\ff(\includegraphics[width=.3cm]{pentagon.eps}\triangle)=(2,0)=(2) .
$ 
}}
%

%
\newcommand{\num}{c}
%
%
A further decomposition of $A_n(\lambda)$ is into subsets
characterised as the inverse images of $\ff$.
Thus, for $\nu$ in $\PP$:
\[
A_{\lambda,\nu} = \ff^{-1}(\nu) \cap A_n(\lambda)
\]
and we have $\AE_{\lambda,\nu}$ similarly. We write 
$a_{\lambda,\nu}$ and $\ae_{\lambda,\nu}$ for the sizes of these sets. 

\mdef \label{rem:c-lambda}
The utility of this decomposition for us lies in the fact that the
`overcount factor' in counting $A_n$ instead of $\AE_n$ 
is uniform 
in $A_{\lambda,\nu}$ --- since every tiling in this subset
belongs to a class of size 
\beq \label{eq:overco}
\num^\nu  = \prod_i C_{\nu_i}
\eq
where $C_n$ is the $n$-th Catalan number: $C_1 = 1$, $C_2 = 2$ and so
on.
That is
\[
\ae_{\lambda,\nu} = \frac{a_{\lambda,\nu}}{\num^\nu }
\]

\newcommand{\meld}{\wr}

$\ $

In order to explain the proof it will be helpful to write out the
first few terms explicitly.
To this end let
\[
\lambda \meld \mu \; = \; \lambda \cup \mu^+ \cup 1^{-|\mu^+|}
\]
To obtain \type\ $ \lambda \meld \mu $ we start with  $\lambda$,
but move to a \type\ with fewer
triangular tiles,  replacing multiple triangles with larger tiles such that the
overall \weight\ is unchanged. 
Here we note that on the level of tilings, 
the operation $\coup(1^{-a})$ corresponds to  
removing $a$ triangles. 

Then:
\[
\ae_n(\lambda) 
 = a_n(\lambda) \Pip{\lambda}{\emptyset} 
   - a_n(\lambda \cup (2) \cup (1^{-2})) \Pip{\lambda}{(1)}
   + a_n(\lambda \cup (3) \cup (1^{-3}))  \Pip{\lambda}{(2)}
   - \ \ldots \hskip2cm
   \]
   becomes
\[
\ae_n(\lambda) 
 = a_n(\lambda) 
   - a_n(\lambda \meld (1)) \Pip{\lambda}{(1)}
   + a_n(\lambda \meld (2))  \Pip{\lambda}{(2)}
   - a_n(\lambda \meld (1^2)) \Pip{\lambda}{(1^2)}
   - a_n(\lambda \meld (3)) \Pip{\lambda}{(3)}
   + ...
\]
We interpret 
each of the terms 
as counting a certain subset of
$A_n(\lambda)$. 
The claim is that the multiplicity of 
the intersection of this subset with $\ff^{-1}(\nu)$ is, for each $\nu$, 
uniform in  $\AE_n(\lambda)$, in the sense of our discussion of
$A_{\lambda,\nu}$ in (\ref{rem:A-l-m}) above. 


\mdef \label{pa:cake}
The set counted by $\Pip{\lambda}{\mu} a_n(\lambda\meld\mu)$ is the subset 
of $A_n(\lambda)$ where the triangulated regions include regions
corresponding to triangulations of the tiles arising from $\mu$
--- specifically the construction is to count tilings of 
$\lambda\meld\mu$
but with the tiles with sizes given by $\mu^+$ marked
(hence the multiplicity $\Pip{\lambda}{\mu} $).
The marked regions can then simply be
triangulated.
In this way we have tilings of the right \type, but where there are
triangulated regions of size at least given by $\mu^+$
(they might also be bigger).
Let us write $B^\lambda_\mu$ for this subset of 
$A_n(\lambda)$. 
(Note that tilings here have only one
triangulation in each triangulated polygon from $\mu$
--- and indeed we do not actually specify one.)
That is
\[
b^\lambda_\mu := | B^\lambda_\mu | = \Pip{\lambda}{\mu} a_n(\lambda\meld\mu)
\]


\subsection{Tilings organised by maximal triangulated regions}

As discussed in (\ref{rem:c-lambda}), if we organise tilings by taking the fibre of $\ff$ 
over some partition $\nu\in\PP$, the overcount factor in counting 
$A_n$ instead of $\AE_n$ is uniform. So  we define the overcount factor 
\beq
\OF_{\mu,\nu} :=  \frac{ |\ff^{-1}(\nu)\cap B^\lambda_\mu | }{|\AE_{\lambda,\nu}|}
\eq
By definition, $\OF_{\nu,\nu}=1$. 
We will need the sizes of these factors, but the first thing to note
is that they do not depend on $\lambda$, since this dependence is
removed by the quotient. 
With this, fibre-wise the claim of Theorem~\ref{thm:general2} becomes
\[
|\AE_{\lambda,\nu}| 
\;\; \stackrel{\begin{small}\mbox{Thm.}\,\ref{thm:general2}\end{small}}{=\joinrel=} \;\;
\sum_{m=0}^{m_1(\lambda)-1}(-1)^m 
\sum_{\mu\vdash m} |\ff^{-1}(\nu)\cap B_{\mu}^{\lambda}|
\]
Dividing by $|\AE_{\lambda,\nu}|$ on both sides, this translates to 

\begin{equation}\label{eq:column-sum}
1  
\;\; \stackrel{\begin{small}\mbox{Thm.}\,\ref{thm:general2}\end{small}}{=\joinrel=} \;\;
\sum_{m=0}^{m_1(\lambda)-1}(-1)^m 
\sum_{\mu\vdash m} \OF_{\mu,\nu} 
\end{equation}

So we will have proved Theorem~\ref{thm:general2} if we can show that every
{\em signed} column sum as in (\ref{eq:column-sum}) equals 1. 


Let us start with a table illustrating the first few cases of the numbers 
$\OF_{\mu,\nu}$ 
--- see Figures~\ref{fig:OF-start},\ref{fig:OF-more}. 
We unpack what this table is saying in a few cases before 
treating the general case.

\ignore{{
\[
\begin{array}{c|cccccccccccc}
      {}_\mu \;\;\; \backslash  {}^\nu \!\!
  & 0 & 1 & 2 &  1^2     &  3       &   21  & 1^3 & 4 &
\\ \hline
a_n(\lambda)
&a_3(1)&a_4(1^2)&a_5(1^3)&a_4(1^2)^2  & a_6(1^4)    &a_5(1^3)a_4(1^2)&a_4(1^2)^3 &a_7(1^5)& \\
&  =1  &  =2   & = 5    &  =2^2      & =  14      &  =   5\times2  & =2^3 &=42&
\\ \hline
B^\lambda_{(1)}  
& 0 &a_4(2)&a_5(21)&  2a_4(1^2)a_4(2) & a_6(21^2) &a_5(1^3)a_4(2)\;\;\; & &a_7(21^3)&
\\
&   &      &      &            &          &\;\;\;+a_5(21)a_4(1^2)& &    &
\\
& 0 &  = 1 &  = 5 &  =2+ 2 &      =21 & =5\times1+5\times2& &  = 84&
\\ \hline
B^\lambda_{(2)}  
& 0 & 0 & a_5(3) &   0       & a_6(31) & a_5(3)a_4(1^2) & & a_7(31^2)&
\\
& 0 & 0 & =1 &   0       & =6   & =1\times2 & &  = 28&
\\ \hline
B^\lambda_{(1^2)}  
& 0 & 0 & 0 & a_4(2)^2      & a_6(2^2) & a_5(21)a_4(2)& & a_7(2^2 1)&
\\
& 0 & 0 & 0 &   =1       & =3  & =5\times1& &  = 28&
\\ \hline
B^\lambda_{(3)}
& 0 & 0 & 0 &  0       &  a_6(4)= 1       & 0 & 
\\
B^\lambda_{(21)}          & 0 & 0 & 0 &  0       &  0      & a_5(3)a_4(2)=1 &
\\
B^\lambda_{(1^3)}          & 0 & 0 & 0 &  0       &  0      & 0 &
\\
B^\lambda_{(4)}          & 0 & 0 & 0 &  0       &  0      & 0 &
\end{array}
\]
}}

\begin{figure}
\[\small{
\begin{array}{|c|c|c|c|c|c|c|c|}
\hline
      {}_\mu \backslash \! {}^\nu \!\!\!
  & \emptyset & (1) & (2) &  (1^2)     &  (3)       &   (21)  & (1^3)  
\\ \hline
\emptyset
&a_3(1)&a_4(1^2)&a_5(1^3)&a_4(1^2)^2  & a_6(1^4)    &a_5(1^3)a_4(1^2)&a_4(1^2)^3  
\\
&  =1  &  =2   & = 5    &  =2^2      & =  14      &  =   5\cdot2  & =2^3 
\\ 
        \specialrule{1pt}{1pt}{1pt}
(1)
&  &a_4(2)&a_5(21)&  a_4(1^2)a_4(2) & a_6(21^2) &a_5(1^3)a_4(2)\;\;\; & 
a_4(1^2)a_4(1^2)a_4(2)  
\\
&   &      &      &       +a_4(2)a_4(1^2)     &          &\;\;\;+a_5(21)a_4(1^2)&  
+ a_4(1^2)a_4(2)a_4(1^2)   
\\
&   &      &      &          &          &   
& + a_4(2)a_4(1^2)a_4(1^2)    \\
& 0 &  = 1 &  = 5 &  =2+ 2 &      =21 & =5\cdot 1+5\cdot 2& =3(2\cdot2\cdot 1)  
\\ 
        \specialrule{1pt}{1pt}{1pt}
(2)
&   &   & a_5(3) &           & a_6(31) & a_5(3)a_4(1^2) &   
\\
& 0 & 0 & =1 &   0       & =6   & =1\cdot2 & 0   
\\ \hline
(1^2)
&   &  &  & a_4(2)^2      & a_6(2^2) & a_5(21)a_4(2) &  a_4(1^2)a_4(2)a_4(2)
\\
&  &  &  &     &   &   &  +a_4(2)a_4(2)a_4(1^2) 
  
\\
&  &  &  &     &   &   &  +a_4(2)a_4(1^2) a_4(2)
 
\\
& 0 & 0 & 0 &   =1       & =3  & =5\cdot1& =3\cdot 2  
\\ 
        \specialrule{1pt}{1pt}{1pt}
(3)
& 0 & 0 & 0 &  0       &  a_6(4)= 1       & 0 & 0  
\\
\hline
(21)          & 0 & 0 & 0 &  0       &  0      & a_5(3)a_4(2)=1 & 0 
\\
        \specialrule{1pt}{1pt}{1pt}
(1^3)          & 0 & 0 & 0 &  0       &  0      & 0 & a_4(2)^3=1 
\\
\hline
(4)          & 0 & 0 & 0 &  0       &  0      & 0 & 0 \\
\hline
\end{array}
}\]
\caption{\label{fig:OF-start} Tabulating the $\OF_{\mu,\nu}$ function.}
\end{figure}

\begin{figure}
\[
\newcommand{\qin}{1.1in}
\begin{array}{|c|c|c|c|c|}
\hline 
      {}_\mu \backslash \! {}^\nu \!\!\!
  & (4) & (31) &(2^2)& (21^2)
\\ \hline
0
& a_7(1^5)& a_6(1^4)  a_4(1^2) \hspace{\qin} &a_5(1^3)a_5(1^3)& a_5(1^3)a_4(1^2) a_4(1^2)
\\
& =42    & 14\cdot 2 \hspace{\qin}         && 5\cdot 2\cdot 2
\\ 
        \specialrule{1pt}{1pt}{1pt}
(1)
& a_7(21^3)& a_6(21^2) a_4(1^2)   + a_6(1^4)a_4(2) & a_5(21)a_5(1^3)
        &a_5(21)a_4(1^2)a_4(1^2)
\\
& =84& 21\cdot 2 \quad + \quad 14\cdot 1 &+a_5(1^3)a_5(21)& +a_5(1^3)a_4(2)a_4(1^2)
\\
&    &                                   &&  +a_5(1^3)a_4(1^2)a_4(2)
\\ 
        \specialrule{1pt}{1pt}{1pt}
(2)
& a_7(31^2)& a_6(31) a_4(1^2)  \hspace{\qin} &a_5(3)a_5(1^3)& a_5(3)a_4(1^2)a_4(1^2)
\\
& =28& 6 \cdot 2 \hspace{\qin}              &+a_5(1^3)a_5(3)&
\\ \hline
(1^2)
& a_7(2^2 1)& a_6(2^2) a_4(1^2) + a_6(21^2) a_4(2) &a_5(21)a_5(21)& a_5(21)a_4(2)a_4(1^2)
\\
& =28 & 3 \cdot 2 \quad + \quad 21\cdot 1    && +a_5(21)a_4(1^2)a_4(2)
\\ 
        \specialrule{1pt}{1pt}{1pt}
(3)
& a_7(41)  & a_6(4) a_4(1^2)  \hspace{\qin}  &0& 0
\\
& =7   & =1\cdot 1  \hspace{\qin} &&
\\
\hline
(21)
& a_7(32) &   \hspace{\qin}    a_6(31) a_4(2)  &a_5(3)a_5(21)& a_5(3)a_4(2)a_4(1^2)
\\
& =7  &    \hspace{\qin} =6 \cdot 1          &+a_5(21)a_5(3)& +a_5(3)a_4(1^2)a_4(2)
\\
\hline
(1^3)
& 0 &   \hspace{\qin}    a_6(2^2) a_4(2)  &0& a_5(21)a_4(2)a_4(2)
\\
&   &    \hspace{\qin} =3 \cdot 1  &&
\\
        \specialrule{1pt}{1pt}{1pt}
(4)
& a_7(5)  & 0 & 0 &  0 
\\
& =1  &  & &
\\
\hline
(31)
& 0  &   \hspace{\qin}   \ a_6(4) a_4(2)  & 0 & 0 \\
&  &    \hspace{\qin} =1        & & 
\\
\hline
(2^2)  & 0 & 0 & 1 & 0
\\
\hline
(21^2) & 0 & 0 & 0 & 1
\\
\hline
\end{array}
\]
  \caption{\label{fig:OF-more} Continuing the $\OF_{\mu,\nu}$ table.
  }
\end{figure}


\subsection{Overcount factor columnwise: small $\nu$} 


$\ $ 

{\bf First column:}
The $\mu$-th entry in the first column is the overcount factor for the subset of
$B^\lambda_\mu$  
(any $\lambda$)
with no two triangular tiles together.
Zero means that in fact the intersection is empty.
In the first column we are looking only at tilings in which the
triangulation has no clusters of size bigger than 1.
Since the equivalence
classes are singletons in this case, the count is the same in
$\AE_n(\lambda)$:
\[
\OF_{0,0} = 1
\]
All of the remaining combinatorics count sets in which higher tiles
are triangulated, so there are larger clusters, and so the
intersection in the first column is empty.
This means that the theorem counts these classes correctly.

{\bf Second column:}
In the second column, $\nu=(1)$, we are intersecting with tilings having a
quadrilateral triangulated region and then any other triangles
singletons.
Such tilings are present in $A_n(\lambda)$
(the first row)
and, since the equivalence
classes are flip pairs in this case, the count is 2x that in 
$\AE_n(\lambda)$.
In the second row (the case of $B^\lambda_{(1)}$)
we assemble tilings with a marked quadrilateral (which
we consider to then be triangulated). However since we only count one tiling
here, we count for $\AE_n(\lambda)$ correctly.
(In this set there are tilings where the triangulated quadrilateral is adjacent
to other triangles, but these do not lie in the intersection in this
column.) 
All of the remaining combinatorics count sets in which higher tiles
are triangulated, so there are larger clusters, and so the
intersection in the second column, for $\mu$ all other $\mu$, is empty.
This means that the theorem again counts these classes correctly overall.

{\bf Column $\nu=(2)$:}
Here we have tilings with a pentagonal triangulated
region and any other triangles singletons. 
Thus the overcount in $A_n(\lambda)$ is $5\times$.
In $B^\lambda_{(1)}$ such pentagons arise when the triangulated quadrilateral 
is adjacent to one other triangle. Indeed note that a given
triangulated pentagonal subregion of the tiling could arise in 5
different ways in $B^\lambda_{(1)}$ --- we are counting the number of
ways of tiling a pentagon with a quadrilateral and a triangle,
which number is $a_5(21)$.
In $B^\lambda_{(2)}$ such pentagons arise when we triangulate a
pentagon, with no other adjacent triangles. The construction only
counts one such triangulation, so the count is the same as for the
class in $\AE_n(\lambda)$.
The remaining intersections in this column are empty.
So again the count in the Theorem is correct.

{\bf Column $\nu=(1^2)$:}
Here we have tilings with two triangulated quadrilaterals and any
other triangles singletons.
The overcount factor in $A_n(\lambda)$ is $2^2$.
In $B^\lambda_{(1)}$ tilings of the required clustering arise when the
triangulated quadrilateral is not adjacent to any triangle, and there are two
other triangles that are adjacent (possibly along with further
singletons).
There are two ways the quadrilateral could be asigned to one of the
triangulated quadrilaterals; and there are two triangulations of the other
triangulated quadrilateral, so we have a 2x2x overcount here. 
In $B^\lambda_{(2)}$ there is a triangulated pentagon, and hence no
intersection.
In $B^\lambda_{(1^2)}$ there are two marked quads. Each is given a
triangulation, so the count is the same as for classes.
There are no more non-empty intersections in this column. 
Again the theorem is verified.

{\bf Column $\nu=(3)$:}
Here we have tilings with one triangulated hexagon.
Such tilings are overcounted in $A_n(\lambda)$ by the appropriate
Catalan number, 14.
In $B^\lambda_{(1)}$ we have the triangulated hexagon formed by a quadrilateral and
two triangles. The number of ways of doing this is $a_6(21^2)$. 


\subsection{Overcount factor columnwise: general $\nu$} $\;$




\newcommand{\cuu}{\cup}
\newcommand{\PPP}{\PP^*}

Let $\PPP = \PP\setminus\{ \emptyset \}$.  
We write $\PP^+$ for the image of $\PP$ under $\mu \mapsto \mu^+$. 
That is $\PP^+ = \{ (0), (2), (3), (2^2), (4), (32), (2^3), ... \}$.
Recall that $\mu \mapsto \mu^-$ acts as the inverse of  
$\mu \mapsto \mu^+$ on $\PP^{*+}$.
% ...

\mdef \label{mdef:models} 
  For $s \ge 1$ consider the set
  of multipartitions 
$\PP^s =
  \{ \gamma = (\gamma^1, \gamma^2, ..., \gamma^{s}) \; | \; \gamma^i \in \PP \}$.
Define a map $\cuu_s : \PP^s \rightarrow \PP$ by
  $\gamma \mapsto \cup_{i=1}^s \gamma^i$.
  
  Since $\cuu_s$ is surjective, the fibres
$\{ \cuu_s^{-1}(\mu) \; | \; \mu\in\PP \}$ are a partition of $\PP^s$. 

Define
$
\modll{s} \mu \; = \; \cuu_s^{-1} (\mu^+).
$
For $\nu\in\PPP$ we write $\modll{\nu} \mu$ for
$\modl{\nu}\mu$.
Thus
\beq \label{eq:ppp}
\bigsqcup_{\mu\in \PP} \modll{\nu} \mu \; = \; \left( \PP^+ \right)^{\len{\nu}}
\eq
is a partition of
$\left( \PP^{\len{\nu}} \right)^+ 
 = \left( \PP^+ \right)^{\len{\nu}}  $.

We write $\gamma\modll{s} \mu$ for  $\gamma\in\modll{s} \mu$.
\ignore{{
  For $s \ge 1$
we write $\gamma\modll{s} \mu$  
to denote that 
$\gamma = (\gamma^1, \gamma^2, ..., \gamma^{s}) \in \PP^s$
such that $\cup_i \gamma^i =$ $\mu^+$.
%
We extend to $s=0$ by taking $\PP^0 = \{ \emptyset \}$. 
}}


\noindent
Examples:
$\modll{2} (21) = 
\{ ((32),\emptyset), ((3),(2)), ((2),(3)), (\emptyset,(32)) \}.
$
%
\newcommand{\nug}{(21^2)}
\newcommand{\nugg}{3} 
And
for $\nu = \nug$ the partitioning of
$\left( \PP^{\len{\nu}} \right)^+$
starts with 
\\
$\modll{\nugg} \kern .05em\emptyset = \{ (0,0,0) \}$; 
\hspace{.32in} 
$\modll{\nugg} (1)  = \{ ((2),0,0),(0,(2),0), (0,0,(2))\}$; 
\\
$\modll{\nugg} (2)  = \{ ((3),0,0),(0,(3),0), (0,0,(3))\}$; 
\\
$\modll{\nugg} (1^2)  = 
\{ ((2^2),0,0),(0,(2^2),0), (0,0,(2^2)), 
  ((2),(2),0),((2),0,(2)),(0,(2),(2)) \}$; 
etc.

\newcommand{\tricomp}{triangular completion}
\newcommand{\Tricomp}{Triangular completion}
\newcommand{\nup}{{\nu^+}}
\newcommand{\fillnu}[1]{\fillup{#1}^{\nup}}
\newcommand{\edible}{additive}

\newcommand{\rr}{|\gamma|} 

\begin{defn}
  For $m \in \N$ and $\gamma\in \PP$    
  define $ \fillup{\gamma}^m   \in\PP_{\geq m} = \cup_{r=m}^\infty \PP_r$ by
$$
\fillup{\gamma}^m \;  =
\left\{ \begin{array}{ll} \gamma\cup 1^{m-\rr} & \mbox{if } \rr \leq m \\ \gamma & \rr>m
\end{array} \right. .
$$
We understand $a_n(\fillup{\gamma}) = a_n(\fillup{\gamma}^{n-2})$
(completing the partial shape $\gamma$ with triangles).
\end{defn}





\begin{lm} \label{lm:overcount}
For arbitrary $\mu \in \PP$,  $\nu \in \PP^*$, 
the overcount function 
$\OF_{\mu,\nu}$ obeys
\beq \label{eq:OF}
\OF_{\mu,\nu} = 
\sum_{\gamma\stackrel{\ \len{\nu}}{\models} \mu} 
\prod_{i=1}^{\len{\nu}} a_{\nu_i +3}(\fillup{\gamma^i})
\eq
\end{lm} 

\begin{ex}
For a full example we consider $\mu=(21)$, $\nu=(31)$
(so $\len{\nu} =2$):  
\[
\OF_{(21),(31)} = a_6(\fillup{32}) a_4(\fillup{\emptyset}) + 
 a_6(\fillup{3}) a_4(\fillup{2}) + a_6(\fillup{2}) a_4(\fillup{3}) +
   a_6(\fillup{\emptyset}) a_4(\fillup{32})
   \;\;  = \;\;   a_6(31) a_4(2)
\]
\\
1. 
$a_6(\fillup{32}) = 0$ 
as the partition $(32)$ has size $5\ge 6-1$, there is no way to 
tile a hexagon with a pentagon and a quadrilateral, while 
$a_4(\fillup{\emptyset})=a_4(1^2)$. 
\\
2. 
$a_6(\fillup{3}) = a_6(31)$ --- we complete the tiling with a triangle.
\\
3. $a_4(\fillup{3}) = 0$ similar to 1., while $a_6(\fillup{2})=a_6(21^2)$. 
\\
4. $a_6(\fillup{\emptyset}) = a_6(1^4)$, again completing with triangles; 
while $a_4(\fillup{32}) = 0$ similar to 1.
\end{ex}


\begin{proof}[Proof of Lemma~\ref{lm:overcount}]
Observe that in each product 
$\prod_{i=1}^{\len{\nu}} a_{\nu_i +3}(\fillup{\gamma^i})$ 
each factor addresses one
 maximal triangulated region (for these regions their sizes, as polygons, are 
the $\nu_i+3$). 
We have to take the maximal subregions prescribed by $\nu$ 
(note that 
we can, and this does, ignore those that are already simple triangles)
and populate them according 
to the polygons available from $\mu$. 
The contribution to the overcount factor is $a_{\nu_i +3}(\fillup{\gamma^i})$ 
since this counts the number of
ways of tiling the subregion as indicated by 
$\gamma^i$
(after this, in our construction, recall from (\ref{pa:cake}), 
we simply triangulate to any one triangulation).  
Note that the sum over 
$ \gamma \modl{\nu} \mu $ 
constructively exhausts all possibilities. 
Indeed it does not check against internal details of $\nu$, so also
formally includes 
some impossibilities, in general --- the cases where 
$ |\gamma^i |  > \nu_i +1$ for some $i$ and  
the factor 
$ a_{\nu_i +3} (\gamma^i ) = a_{\nu_i +3} (\fillup{ \gamma^i } ) $
is zero.
(Since the extra terms are zero this does not affect the identity.)
\end{proof}


Our next step in proving (\ref{eq:column-sum}) is 
Lemma~\ref{lm:OF-is-prod-F}.
We will need a couple of preliminaries.

\ignore{{
\mdef Consider the composition $(\PP, \cup)$. 
Let $(\Z^\N,+)$ denote the group of integer-valued vectors of 
unbounded length but finite support
(the additive group of integer-valued polynomials). 
Let $\N_0^\N$ denote the submonoid of non-negative
valued vectors. 
Expressing $\PP$ in the exponential form we have 
$\PP \cong \N_0^\N$ and indeed 
$(\PP,\cup) \cong (\N_0^\N , +)$. 
}}

\ignore{{

\begin{lm} \label{lm:overcountx}
  Fix  $\nu \in \PP^*$ \ppm{I restricted here! omitted case is already
  done anyway!} and 
let $\star()$ denote an arbitrary \edible\ (\kub{additive?}) function
of $\PP^{\len\nu}$.
Then  
\[
\sum_{\mu \in\PP} \sum _{\gamma \modll{\nu} \mu} \star(\fillup{\gamma}^{{\nup}} ) 
 = \sum_{\delta \in \PP_{\geq \nu^+} } \star(\delta ) 
\]
\end{lm}

\ppm{
?}
\kub{too tired now to fully understand. do you mean the step after `Finally we need to show' 
is not correct and can't be true? - and quality control?!} \ppm{yes;checking}

\ppm{Can we say it as follows?:}

\newpage 
}}

\mdef
Note that for any $m$ and $\gamma \in \PP$ 
\beq \label{eq:key1}
\gamma \mapsto \fillup{\gamma^+}^m \mapsto \left(  \fillup{\gamma^+}^m \right)^- = \gamma
\eq
Meanwhile $\gamma \mapsto \gamma^-$ is defined and injective on
$\PP_r$ for $r>0$ 
by (\ref{eq:PiP}), and then
\beq \label{eq:keyi}
\gamma \mapsto \gamma^- \mapsto \fillup{\left( \gamma^- \right)^+}^{r}
= \gamma
\eq

\mlem{ \label{lem:rw1}
For $m\in\N$ the map $\gamma \mapsto \fillup{\gamma^+}^m$ is an injection
$\PP \rightarrow \PP_{\geq m}$;  
  and the image includes $\PP_m$.
}
\proof{ First note that $\fillup{\gamma^+}^m \in \PP_{\geq m}$ by
construction.
Then note that injectivity follows from (\ref{eq:key1}).
Finally the image property follows from (\ref{eq:keyi}).
\qed}

%
%

\medskip



\mdef
For $\nu\in\PP$ define $\PP_{\nu} = \times_i \PP_{\nu_i}$; and
$\PP_{\geq \nu} =   
 \{ \delta \in \PP^{\len\nu} \; | \; |\delta^i | \geq \nu_i \forall i \}  .  $  

Given $\nu\in\PP$ and a multipartition $\gamma \in \PP^{\len{\nu}}$  we define 
$\fillup{\gamma}^\nu = (\fillup{\gamma^1}^{\nu_1},\fillup{\gamma^2}^{\nu_2} , ...)$. 
Note that $\fillup{\gamma}^\nu \in \PP_{\geq \nu}$.


\ignore{{
\mdef
We write $\modll{\nu} \mu$ for the set 
$\{ \gamma \; | \; \gamma \modll{\nu}\mu \}$.
That is, $\modll{\nu} \mu$ is the set of all multipartitions $\gamma$ 
of length
$ \len\nu$ whose 
$\cup_{i=1}^{\len\nu} \gamma^i = \mu^+$
(Def.\ref{de:models}).
}}



\ignore{{
\ppm{Not needed any more!:
(1) The sets $\{ \modll{\nu} \mu \; | \; \mu\in\PP^{} \}$ 
are disjoint, that is   
$$
\bigsqcup_{\mu\in \PP} \modll{\nu} \mu   \; = \;
  \bigcup_{\mu\in \PP} \modll{\nu} \mu  
$$
and are a partition of 
$\left( \PP^{\len{\nu}} \right)^+ 
 = \left( \PP^+ \right)^{\len{\nu}}  $.
}
}}

\begin{lm} \label{lem:rw2}
  Fix  $\nu \in \PP^*$.
  The map $\gamma \mapsto \fillup{\gamma}^{\nu^+}$ is an injection
$\bigsqcup_{\mu\in \PP} \modll{\nu} \mu  \rightarrow \PP_{\geq \nu^+}$;
 and the image includes $\PP_{\nu^+}$.
\end{lm} 


\newcommand{\aaa}{\alpha}

\begin{proof}
We write the  maps from  \ref{lem:rw1} as
   $\gamma \mapsto \gamma^+ \mapsto \fillup{\gamma^+}^m$
  taking $\PP \rightarrow \PP^+ \stackrel{\aaa_m}{\rightarrow} \PP_{\geq m}$.
By (\ref{eq:ppp}) we have
$\bigsqcup_{\mu\in \PP} \modll{\nu} \mu \; = \; \left( \PP^+ \right)^{\len{\nu}}$.
Thus the map here is the Cartesian product
$\times_{i=1}^{\len{\nu}} \aaa_{\nu^+_i}$.
  Thus it is injective.
  Since the image of a Cartesian product is the product of images, the
  image bound of Lemma \ref{lem:rw1} also implies the image bound
  here.
\end{proof}

\newpage

\begin{lm}\label{lm:OF-is-prod-F}
For every $\nu\in\PP$ we have 
\[
\sum_{\mu\in\PP}(-1)^{|\mu|}\OF_{\mu,\nu}= \prod_{i} F_{\nu_i}
\]
where
\beq \label{de:Fr}
F_r \; := \;  \sum_{m=0}^{r} (-1)^m \sum_{\rho\vdash m}   a_{r+3}(\fillup{\rho^+})
\eq
\end{lm}

\newcommand{\appone}{
\ppm{It might be good temporarily to have an example of \ref{lm:OF-is-prod-F}, so we can line
  up all the notations... Not sure if following is sufficiently generic!?...}
{Example. Case $\nu=(31)$:
$$
F_3 = +a_6(\fillup{\emptyset}) - a_6(\fillup{(1)^+})
+  a_6(\fillup{(2)^+}) +  a_6(\fillup{(1^2)^+})
-  a_6(\fillup{(3)^+}) -  a_6(\fillup{(21)^+}) -  a_6(\fillup{(1^3)^+})
$$
$$ \hspace{.351in}  =  +a_6(1^4) - a_6(\fillup{(2)})
   +  a_6(\fillup{(3)}) +  a_6(\fillup{(2^2)})
   -  a_6(\fillup{(4)}) -  a_6(\fillup{(32)}) -  a_6(\fillup{(2^3)})
$$
$$ \hspace{.351in}  =  +a_6(1^4) - a_6({(21^2)})
+  a_6({(31)}) +  a_6({(2^2)})
-  a_6({(4)}) - \underbrace{ a_6(\fillup{(32)}) -  a_6(\fillup{(2^3)})}_{=0}
$$
and
$$
F_1 \; = \;  +a_4(\fillup{\emptyset}) - a_4(\fillup{(1)^+})
    \; = \;  +a_4({\emptyset}) - a_4(\fillup{(2)})
    \; = \;\;  +a_4(1^2) - a_4({(2)})
$$
Meanwhile
\[
\sum_{\mu\in\PP} \pm  \OF_{\mu,(31)}  \;\;
  = \; \OF_{\emptyset,(31)} \;\; \pm \OF_{(1),(31)} \;\;
    \pm \OF_{(2),(31)} \pm \OF_{(1^2),(31)} \hspace{1.5in}
\] \[
\pm \OF_{(3),(31)} \pm \OF_{(21),(31)} \pm \OF_{(1^3),(31)}
\hspace{1in} \]
\[
 \pm \OF_{(4),(31)} \pm \OF_{(31),(31)} \pm \OF_{(22),(31)} \pm \OF_{(21^2),(31)} \pm \OF_{(1^4),(31)} \pm
    ...
\] \medskip
\[
= \;\; +a_6(1^4) a_4(1^2) \;\; - a_6(21^2) a_4(1^2) - a_6(1^4) a_4(2) \;\;\hspace{1cm}
\] \[
 + a_6(31) a_4(1^2) + \underbrace{ a_6(1^4) a_4(3) }_{=0}
 \]\[
 + a_6(22) a_4(1^2) + a_6(21^2) a_4(2)  + \underbrace{ a_6(1^4) a_4(22) }_{=0}  \;\;
 \]\[
 +a_6(4) a_4(1^2)  + \underbrace{ a_6(1^4) a_4(4) }_{=0}  \;\;
 + ...
\]
Here the injection is that a pair
$(\delta,\gamma) \in \PP_{\nu_1 +1} \times \PP_{\nu_2 +1}$
is taken to  ....same? Need to show that same lies in target?

Claim: the map 
$\PP_a \times \PP_b \rightarrow ???$ given by 
$(\delta,\gamma) \mapsto (\delta^-,\gamma^-)$
is well-defined and injective...

}

}

\begin{proof}
Note that $a_{r+3}(\rho)=0$ when $|\rho| > r+1$. Applying 
Lemma~\ref{lem:rw1} to (\ref{de:Fr}) then gives 
$$
F_r = \sum_{\delta\in\PP_{r+1}}  
            (-1)^{|\delta^-|} a_{r+3}(\delta) .
$$

\noindent 
Expanding 
$
\prod_{i=1}^s F_{\nu_i}
=(\sum_{\delta^1\in\PP_{\nu_1+1}} 
  (-1)^{|\delta^{1-}|} a_{\nu_1+3}(\delta^1)) 
\cdots (\sum_{\delta^s\in\PP_{\nu_s+1}}  
  (-1)^{|\delta^{s-}|} a_{\nu_s+3}(\delta^s))
$
we obtain
\beq \label{eq:Fexpand} 
\prod_{i=1}^s F_{\nu_i}
= 
\sum_{(\delta^1,\dots,\delta^s)\in\PP_{\nu_1+1} \times\dots\times \PP_{\nu_s+1}} 
  (-1)^{\sum_i |\delta^{i-}|} a_{\nu_1+3}(\delta^1)\cdots a_{\nu_s+3}(\delta^s) 
\eq
where $s=\len{\nu}$. 
But, noting the vanishing condition again,
applying Lemma~\ref{lem:rw2} to $\sum_\mu \pm \OF_{\mu,\nu}$
yields the same expression up to signs.
So it remains to check the signs. 

\newcommand{\gammad}{\delta}

Let $b:= 
(-1)^{\sum_i |\delta^{i-}|} a_{\nu_1+3}(\gammad^1)\cdots a_{\nu_2+3}(\gammad^s)$ be a summand 
of $\prod_i F_{\nu_i}$.
\footnot{\ppm{Not needed any more?}
By definition, such a summand arises from partitions 
$\rho^1\vdash M_1,\dots, \rho^s \vdash M_s$ 
where $0\le M_i\le \ppm{\nu_i}$ for all $i$.
The sign of $b$ as a term of $\prod_i F_{\nu_i}$ 
is then $(-1)^{M_1+\dots + M_s}$.
Now $\rho^i$ is recovered from $\gammad^i$ by
\ppm{merging all the parts of $\gammad^i$???} and 
inverting the increase function $(-)^+$, i.e. $\rho^i$ is the partition with 
$m_j(\rho^i)=m_{j+1}(\gammad^i)$ for $j\ge 1$. 

With this, $M_i=\sum_{j\ge 1}m_j(\rho^i)=\sum_{j\ge 1}m_{j+1}(\gammad^i)$ 
\ppm{is this correct?}
(a finite sum) and so 
$M_1+\dots + M_s=\sum_{i=1}^s \sum_{j\ge 1} m_{j+1}(\gamma^i)$ 
$=\sum_{i=1}^s \sum_{j\ge 2} m_j(\gammad^i)$.  
}
%
Note 
that $b$ is a non-zero term of $\OF_{\mu,\nu}$ for 
$\mu = \left( \bigcup_i \gammad^i \right)^-$
(the partition with $m_j(\mu)=m_{j+1}(\bigcup_i \gammad^i)$ for $j\ge 1$),
but with sign $(-1)^{|{\mu}|}$. 
Now 
$|{\mu}|=\sum_{j\ge 1} jm_j(\mu) $ 
$ =\sum_{j\ge 1}jm_{j+1}(\coup_i \gammad^i)$. 
Let us compare this with $(-1)^{\sum_i |\delta^{i-}|}$.
We have
$$
|\delta^{i-}| =\sum_{j\geq 1} jm_j(\delta^{i-}) 
    = \sum_{j\geq 1} jm_{j+1}(\delta^{i})
$$
so 
the signs match up. 
\end{proof}


\begin{lm}\label{lm:EP-formula}
For each $r$, $F_{r}=1$. 
\end{lm}
\begin{proof}
Observe from (\ref{eq:triangles-case})  
that $F_r$ is the \EP\ characteristic.
  Now use Theorem~\ref{th:KisP} and Theorem~\ref{th:SP}.
\end{proof}

\begin{proof}[Proof of Theorem~\ref{thm:general2}]
Combining
Lemma~\ref{lm:EP-formula}
and
  Lemma~\ref{lm:OF-is-prod-F} 
gives 
(\ref{eq:column-sum}) as required. 
\end{proof}



\hspace{.3in}

\section{On combinatorial isometries and other discussion points}

This section takes the form of some extended remarks, on open
questions
and possible connections. 

Here we have drawn attention to the classification of cells of the
associahedron by  mechanisms such as  `\type s' of
tilings as defined in (\ref{de:sh}).  
%
One question thus raised  
is
how (literal) shapes of cells
in the associahedron 
are related to 
\type s of tilings. 






Consider the associahedron as a combinatorial complex (see for example \cite{Basak2010}). 
We call a self-map of the (combinatorial) complex
a combinatorial isometry. 
This should be distinguished from an isometry of one of the concrete
polytopal
realisations, since these depend on the realisation. A typical
realisation has {\em no} true isometries.
On the other hand at least some non-trivial self-maps of the complex
exist.
That is, the ones induced by rotating the polygon $P$.
%
We illustrate  by considering  
$K_5$ and $K_6$.

\subsection{Associahedron $K_5$} 

Consider
Figure~\ref{fig:rosetta}
(interpolating between the polytopal and tiling complex
realisations of $K_5$, 
by an extension of Brown's figure
in \cite{Brown09}).
\begin{figure}
\includegraphics[width=2.09in]{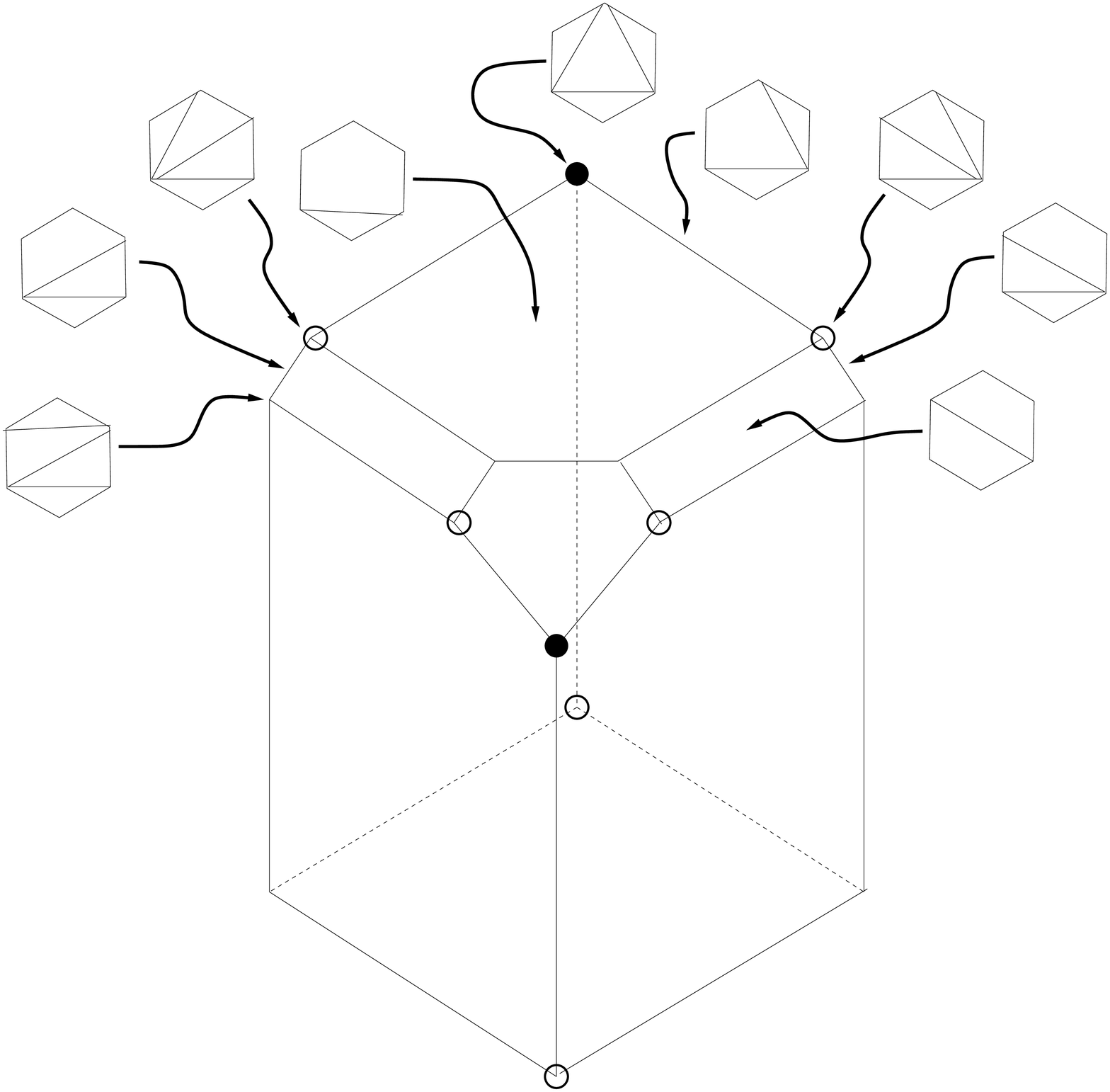}
\raisebox{.41in}{
$
\newcommand{\lin}{.2in} 
\begin{array}{c|cc|c|c}
  \dim = 3 &  2 && 1 & 0 \\ \hline
  \mu    =(3)  & (2) & (1^2) & (1) & (0) \\
  \lambda=(4)  & (31)& (22)& (21^2) & (1^4)  \\
  K_5  \; 
  & \; K_4 \times K_2\;  & \; K_3^2   \;
  & \; K_3 \times K_2^2 \;
  & K_2^4
\\
\includegraphics[width=\lin]{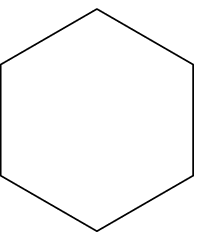} & 
\includegraphics[width=.2in]{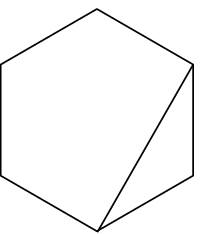} &
\includegraphics[width=.2in]{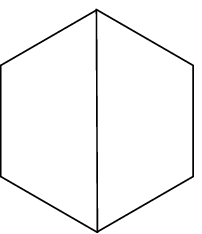} &
\includegraphics[width=.2in]{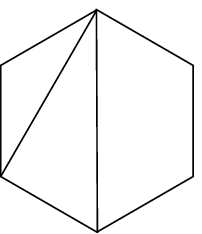} 
& \includegraphics[width=.2in]{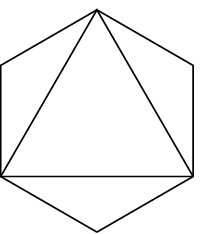} 
 \\
 & & & 
\includegraphics[width=.2in]{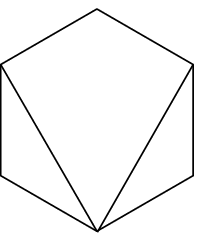}  &
 \includegraphics[width=.2in]{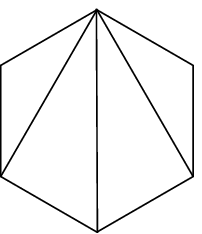}\\
 & & & 
\includegraphics[width=.2in]{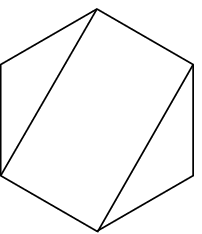} &
\includegraphics[width=.2in]{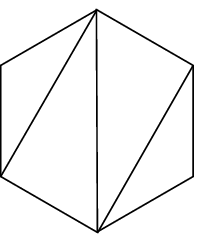}
\end{array}
$}
\caption{\label{fig:rosetta} Facets and shapes in $K_5$.}
\end{figure}
%
Here we see that there are various sub-classifications of $i$-cells
of $K_n$ available.
Firstly
there is the evident and well-known intrinsic classification by products of
smaller associahedra.
This essentially coincides with our tiling shape classification
via $\lambda \leadsto \times_i K_{\lambda_i^+} $.
This corresponds to ignoring
all triangles in tilings, since $K_2$ is trivial.
What about extrinsic properties?

%
Firstly consider the case of 0-cells.
In terms of their intrinsic physical shape these are obviously all the same ---
points. (Compare this with their equivalence under flip in the tiling
realisation.)
From another perspective, 0-cells
are `tristate'  points (points belonging to three cells).
They can be sorted into black circles (`tristate' 
points 
for 3 5-gonal faces); white circles (tristate points for two 5-gonal and
one 4-gonal face that are adjacent to $5^3$ tristate points);
the rest (tristate points of $5^2 4$ type that are not adjacent to
$5^3$ tristate points).
It will be clear from the figure that these classes coincide (in this case) with the classes
induced under rotation of $P$. 

What about 1-cells?
Again these all have the same shape. (They might have different
lengths in a concrete realisation, but this is not canonical.)
We leave it as an exercise to consider extrinsic properties. 


In preparation for looking briefly at $K_6$ (where our 4d sketching skills fail)
we include a tabulation of cells in  Figure~\ref{fig:rosetta} in a format that does
lift to $K_6$. 
The table is organised by dimension of cell.
The $\mu$  
label runs
through the integer partitions of the dimension
and $\lambda=\fillup{\mu^+}$
(let us generally exclude those $\mu$s that give an impossible $\lambda$). 
The final component gives representative tilings, up to 
polygon isometries. 

\ignore{{
\kub{thes following few lines: please, can you check what you want here and maybe explain? 
I can't see why it is obvious (`specifically if...')}
Some of these label cell types, specifically if the number of parts is
equal to codimension + 1.
Thus in particular $(1^3)$ is not possible in this rank.
On the other hand some of these labels correspond to multiple
isometry classes in
the tiling realisation, as indicated. 

\kb{rather say something like:} \\
note that $\mu=(2,1)$ and $\mu=(1^3)$ do not appear as the corresponding $\lambda$ 
do not have the right weight... 
}}

\subsection{Associahedron $K_6$} 

Here drawing a picture is hard. But we
have the tabulated form as follows. 
\ignore{{
describe the associahedron by organising the cells according to their structure 
as products of lower rank associahedra, as in the table for $K_5$. 
\kub{do you want the next bit here?} (and cf. our incidental tiling
classifications in our unpublished work). 
The classification via products is well-known. To consider it with the
tools to hand here, note that the tiling realisation of a cell
partitions $P$ into codimension +1 parts.
The possible further tiling of each part
then corresponds to the contribution of a lower associahedron to the
cell.
For example in codim=1 we have a single diagonal partitioning $P$
into two parts. Thus a corresponding cell is a product of two lower
hedrons (one of which may be `trivial', i.e. $K_2$). 
}}
\ignore{{
$
K_6  \; | \; K_5 \; , \; K_4 \times K_3 \; , \; K_3^3 \; | \; K_4 \; ,
\; K_3^2 \; | \; K_3 \; | K_2 
$
}}

$$
\newcommand{\lin}{.2in} 
\begin{array}{c|ccc|cc|c|c}
  \dim = 4 & 3 & & & 2 && 1 & 0 \\ \hline
  \mu    =(4) & (3) & (21) & (1^3) & (2) & (1^2) & (1) & (0) \\
  \lambda=(5) & (41)& (32) &       & (31^2)& (221)& (21^3) & (1^5)  \\
  K_6  \;
  & \; K_5\times K_2 \; , & \; K_4 \times K_3 \; , & \; K_3^3 \;
  & \; K_4 \times K_2^2\; , & \; K_3^2 \times K_2 \;
  & \; K_3 \times K_2^3 \;
  & K_2^5
\\
\includegraphics[width=\lin]{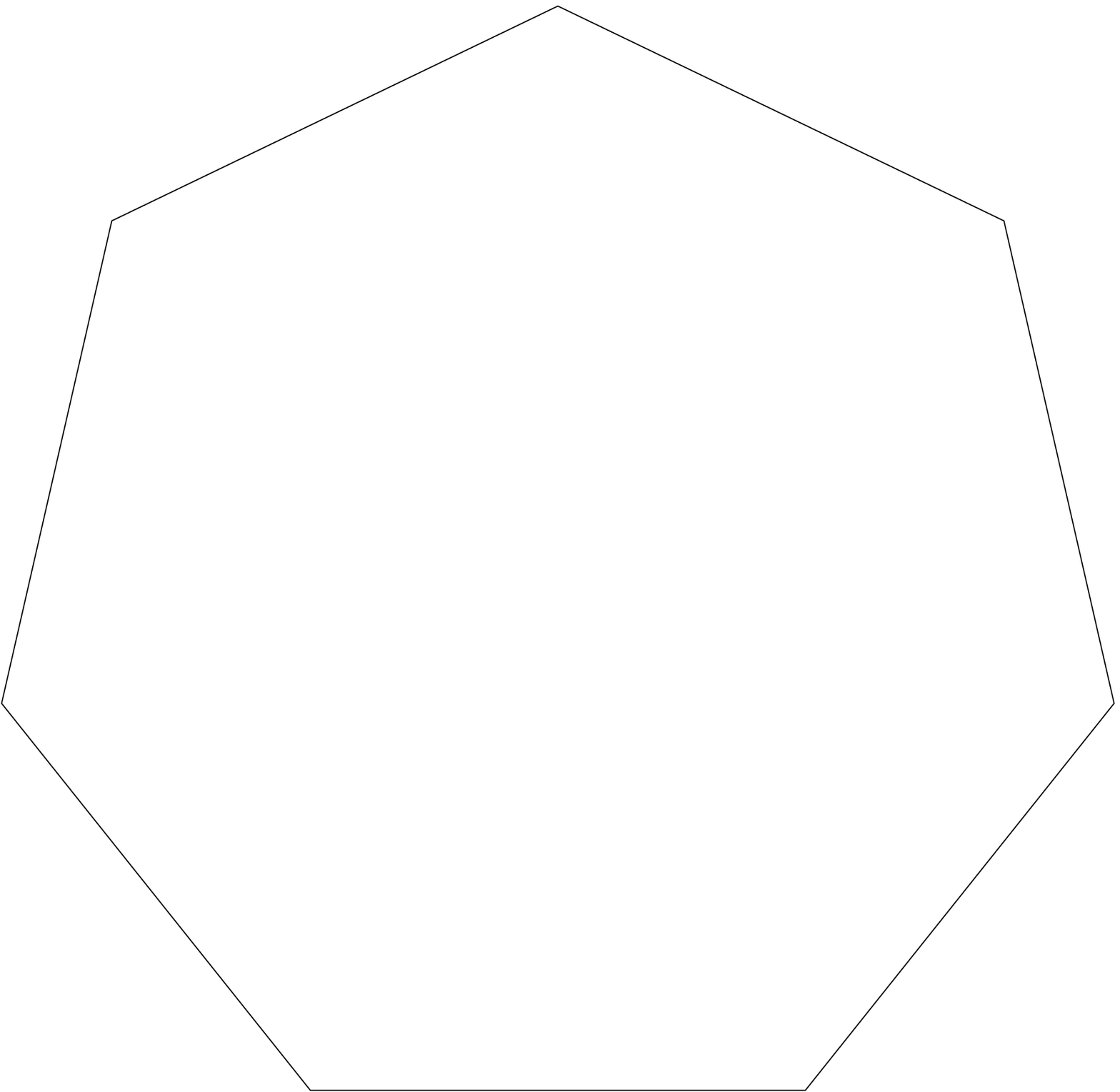} & 
\includegraphics[width=.2in]{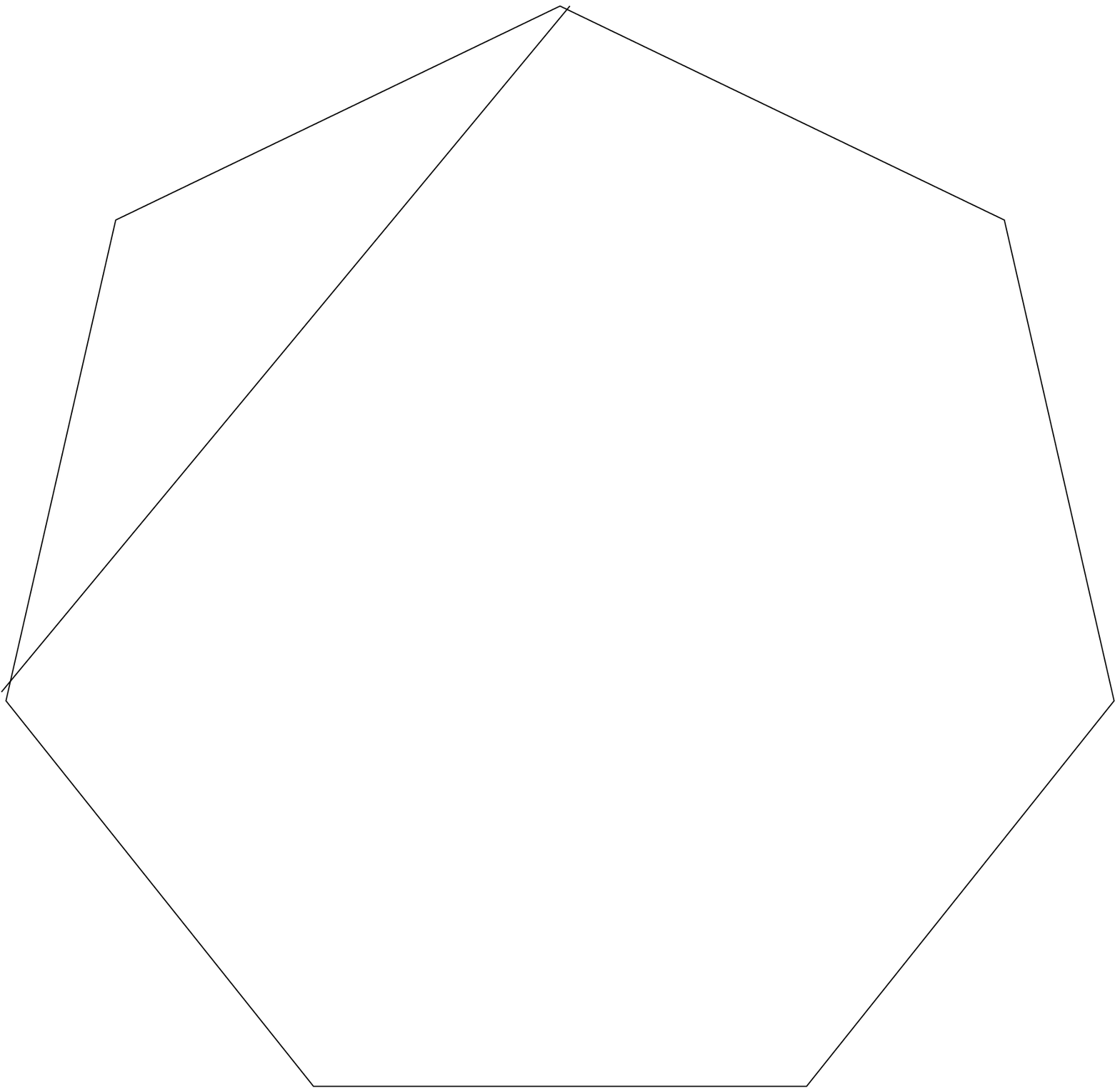} & 
\includegraphics[width=.2in]{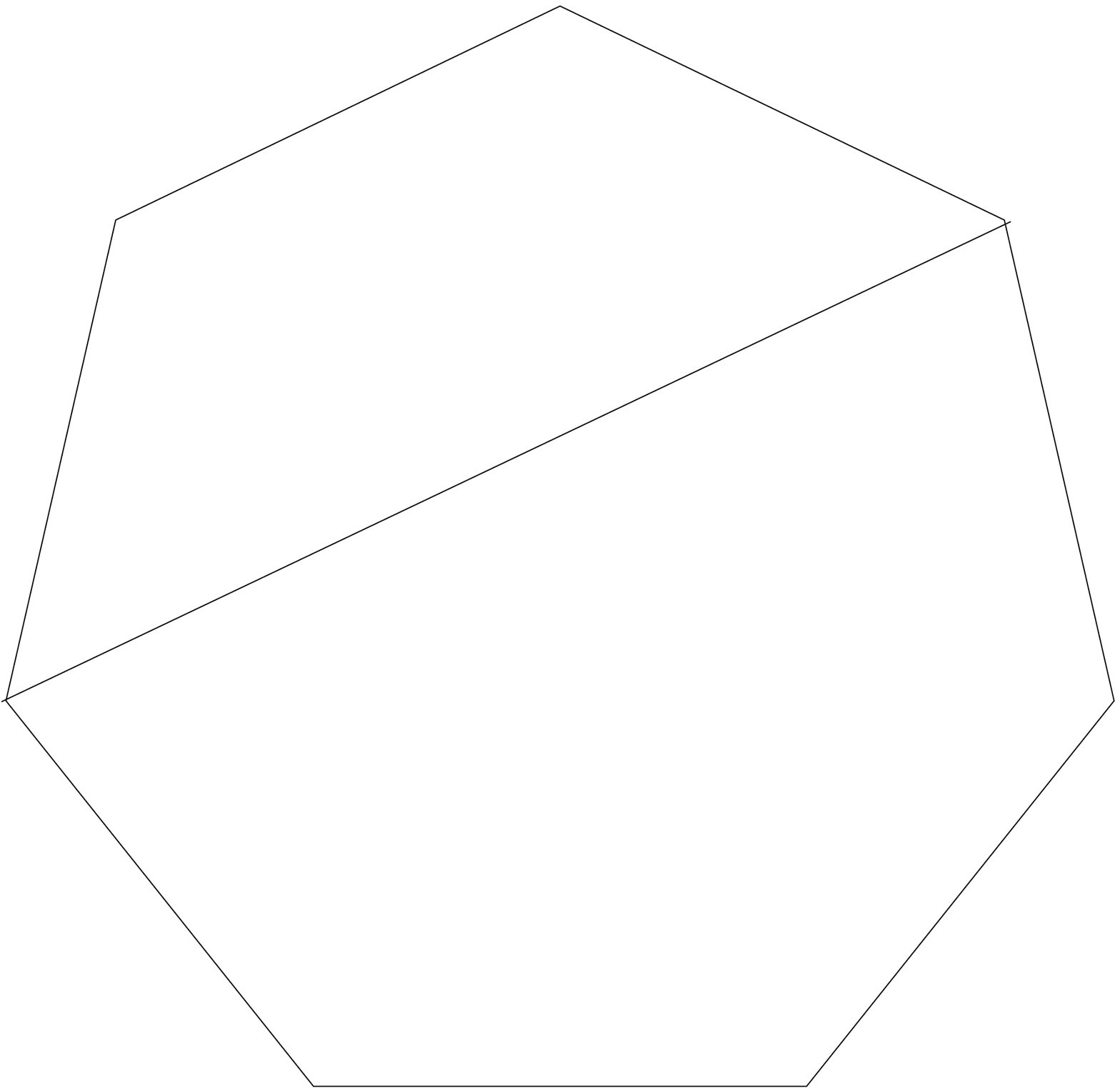} &
-&
\includegraphics[width=.2in]{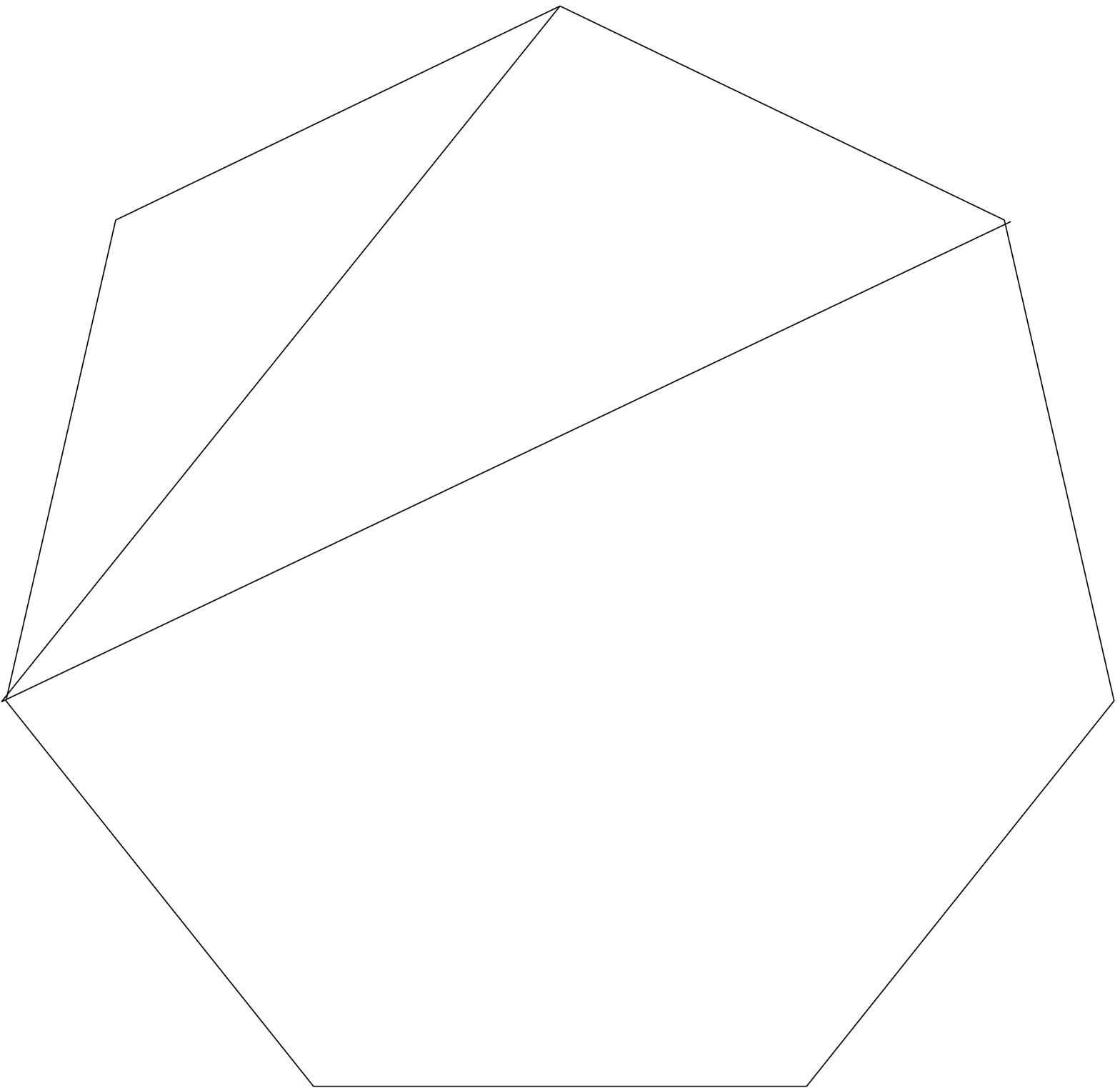} &
\includegraphics[width=.2in]{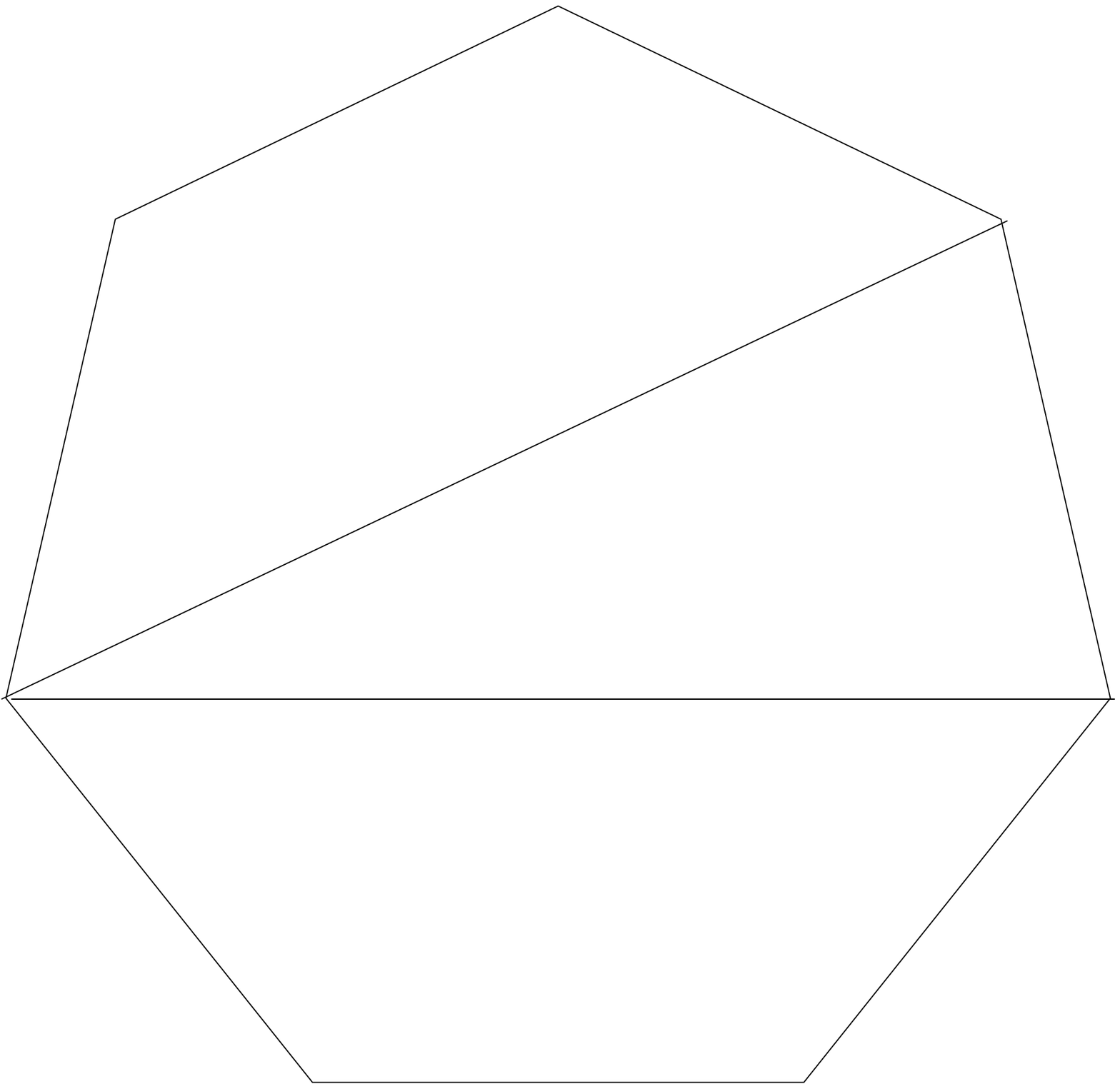} &
\includegraphics[width=.2in]{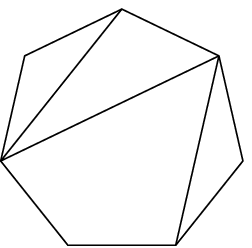}
& \includegraphics[width=.2in]{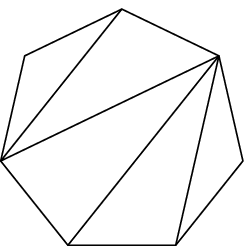} \\
& & & &
\includegraphics[width=.2in]{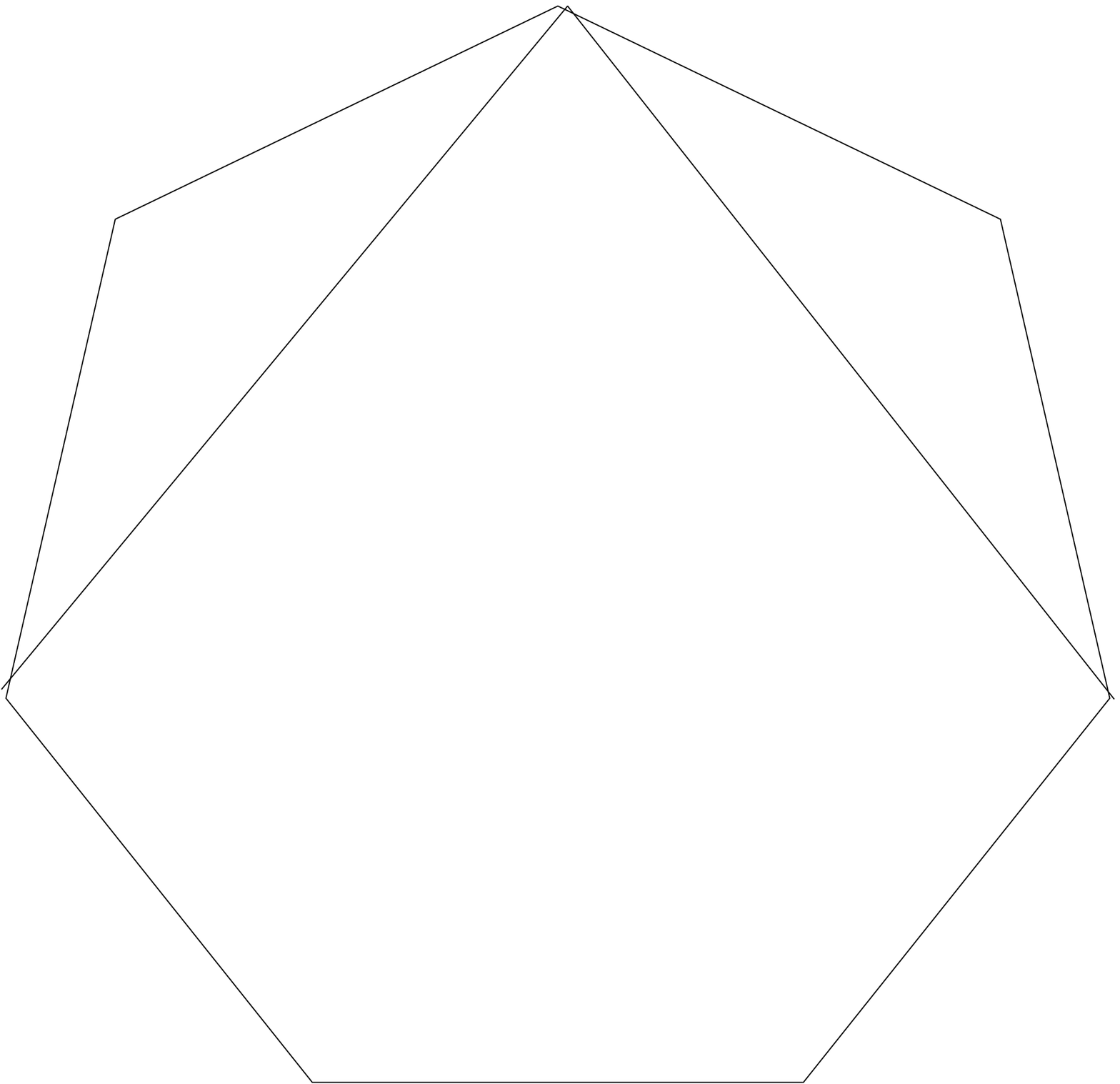} &
\includegraphics[width=.2in]{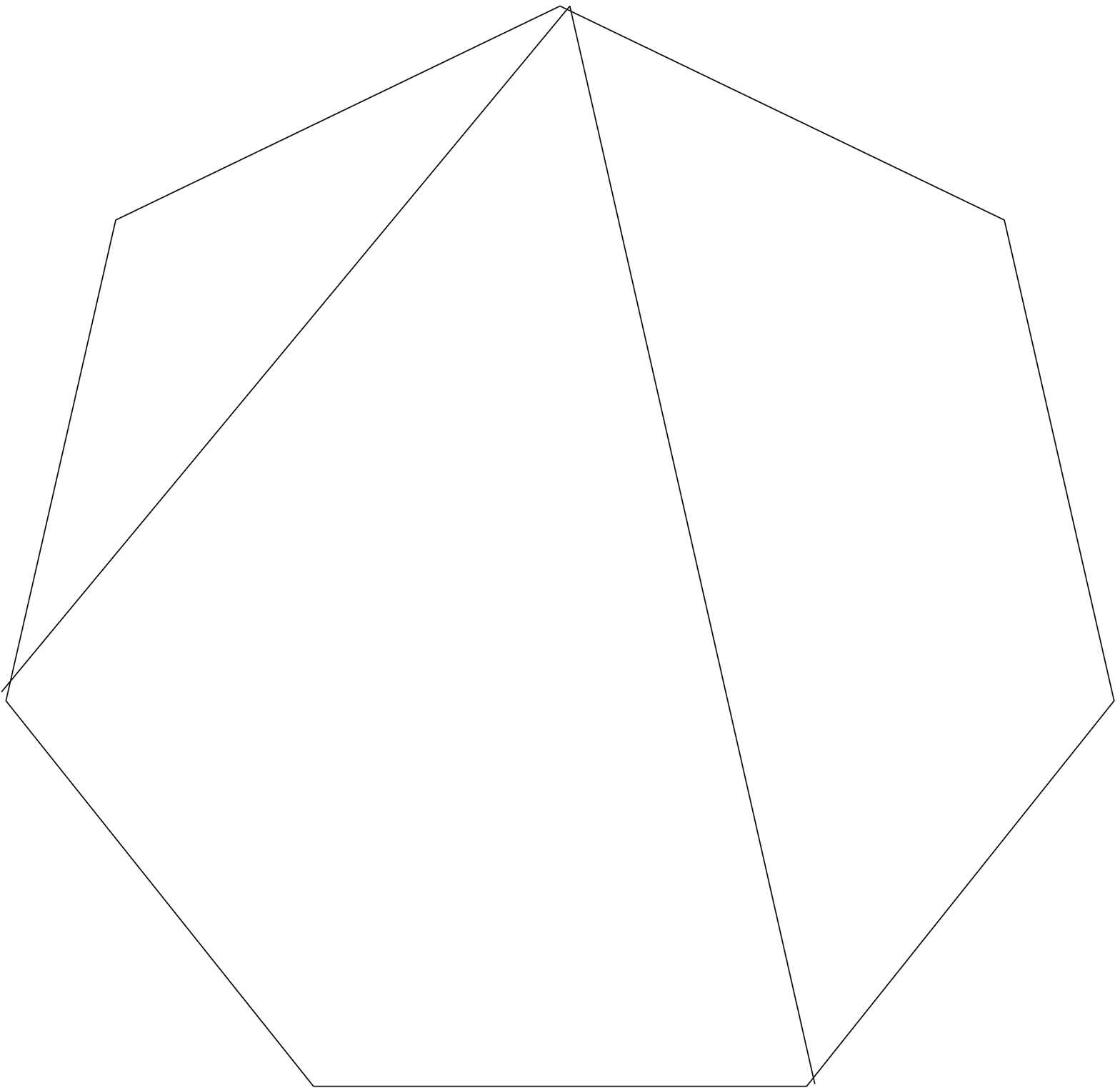} & \vdots  & 
\vdots \\ 
&&&& \includegraphics[width=.2in]{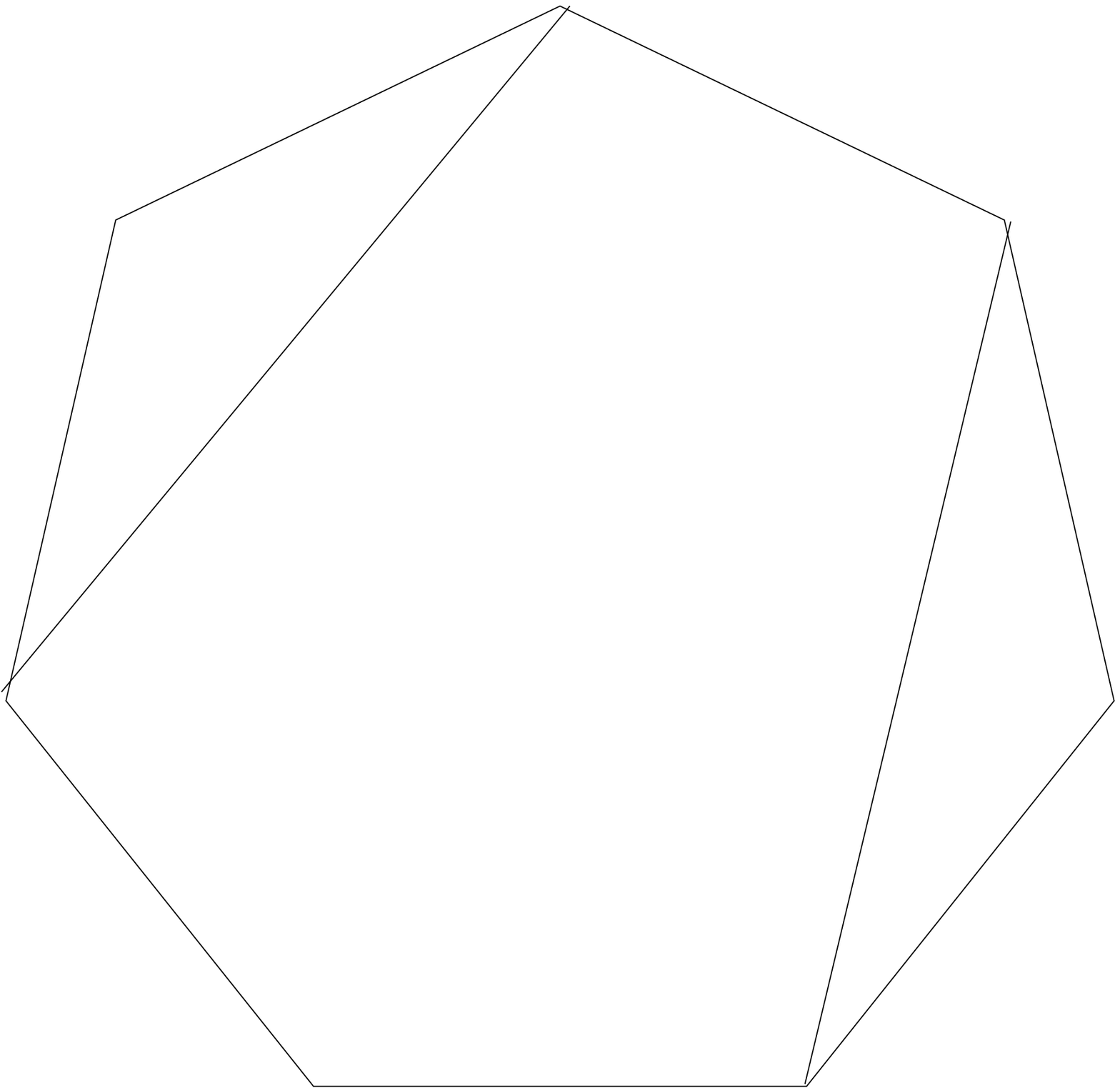} && 
\end{array}
$$


\noindent
NB although we include the column here, $\mu=(1^3)$ is not possible in
this rank, since $|\lambda| = | (2^3) |  = 6>n-1$.
On the other hand some $\mu$ labels again correspond to multiple
isometry classes in
the tiling realisation, as indicated.

\subsection{Brief remarks on connections with other areas} 

General connections of associahedra to several areas are already
mentioned in \S\ref{ss:intro}.
More specifically here, there are a number of areas of representation theory where the map
$\mu \mapsto \mu^-$ plays an interesting role.
See for example the partition algebra \cite{MartinSaleur94b,Jones94}
(and hence geometric complexity \cite{Mulmuley06});
and Gamba's formula \cite[CH.\,VI\,\S 4]{Boerner70}.
For connections to 
moduli spaces see e.g. \cite{DevORourke,Brown09}. 
For 
cyclic sieving see e.g. \cite{ppr}.
For
Baur et al.'s tiling/Temperley--Lieb correspondence see \cite{AABV2017}.
Finally here we mention that there are potential connections to 
quantum codes via 2d surface tilings (for a review see e.g. \cite{Breuckmann16}).


\subsection*{Acknowledgements}
KB thanks Max Glick and Lukas Andritsch for 
discussions. 
KB acknowledges support from FWF grants P30549-N26 and DK W1230. 
\\ 
PM thanks Robert Marsh and Joao Faria Martins for discussions; and 
acknowledges support from EPSRC grant EP/I038683/1 ``Algebraic, Geometric
and Physical underpinnings of Topological Quantum Computation''.

\appendix

\bibliographystyle{amsplain}
\bibliography{bm}

\end{document}